\documentclass[journal,twoside,web]{IEEEtran}

 \usepackage{balance}   
\usepackage{mathtools}
\usepackage{xcolor} 
\usepackage{amsmath,amssymb,bbm}
\usepackage{amsthm}

\usepackage{cite}
\usepackage{enumitem}
\usepackage{cancel}
\usepackage{graphicx}
\usepackage{subcaption}
\usepackage{pifont}
\usepackage{xspace}
\usepackage{makecell}
\usepackage{mathrsfs}
\usepackage{mathtools}
\newtheorem{assumption}{Assumption}
\newtheorem{lemma}{Lemma}

\newtheorem{remark}{Remark}

\newtheorem{theorem}{Theorem}

\newtheorem{definition}{Definition}

\newcommand{\CommentState}[1]{\Statex\hspace{\algorithmicindent}{\color{blue}// #1}}

\newcommand{\calS}{\mathcal{S}}
\newcommand{\calA}{\mathcal{A}}
\newcommand{\calO}{\mathcal{O}}
\newcommand{\calP}{\Delta} % probability simplex over a set
\newcommand{\Nat}{\mathbb{N}}

\allowdisplaybreaks % allow align environment to distribute over multiple pages

\usepackage{algorithm}
\usepackage{algpseudocode}
\algrenewcommand\algorithmicrequire{\textbf{Input:}}
\algrenewcommand\algorithmicensure{\textbf{Output:}}
\makeatletter\@mparswitchfalse\makeatother\normalmarginpar % to force the notes to the right
\usepackage[textwidth=1cm, textsize=scriptsize]{todonotes}

\def\BibTeX{{\rm B\kern-.05em{\sc i\kern-.025em b}\kern-.08em
    T\kern-.1667em\lower.7ex\hbox{E}\kern-.125emX}}
\title{A Decentralized Partially Observable Team Decision Methodology with Delayed Information Sharing}

\author{%
	Xiaoxing~Ren, Thomas~Parisini,  Andreas A. Malikopoulos
\thanks{Xiaoxing Ren is with the School of Civil and Environmental
Engineering, Cornell University, Ithaca, NY,
USA. E-mail: {\tt \small xr49@cornell.edu} }%
 \thanks{Thomas Parisini is with  the Dept. of Electrical and Electronic Engineering,
Imperial College London, London SW7 2AZ, UK, and also with the Dept. of Electronic Systems, Aalborg University, Denmark, and with the  Dept. of Engineering and Architecture, University of Trieste, Italy. E-mail: {\tt\small t.parisini@imperial.ac.uk}}
     \thanks{Andreas A. Malikopoulos is with Cornell Robotics, Center for Applied Mathematics, Systems Engineering program, School of Electrical \& Computer Engineering, Sibley School of Mechanical \& Aerospace Engineering, and the School of Civil \& Environmental Engineering, Cornell University, Ithaca, NY, USA.  E-mail: {\tt\small amaliko@cornell.edu} }
}

\begin{document}

\maketitle

\begin{abstract}
% In this paper, we study the decentralized partially observable team problem under non-classical information sharing. In particular, we focus on the $\gamma$-observable POMDPS, and use representation learning to let members learn a function that approximate the system dynamic. Under 

% We study decentralized partially observed team decision problems with structured latent dynamics. The considered formulation combines a team-theoretic equivalence framework with a low-rank representation of the underlying model. In this setting, each member makes decisions based on its local private data together with delayed common data shared across the team.
% Based on its local private data and the delayed common data, each member constructs an approximate low-rank Markov decision process and applies least-squares value iteration to compute its policy. Under suitable structural conditions, we show that the resulting member-side solution approximates the original centralized team problem. In particular, each member recovers the corresponding component of an approximate team-optimal solution, despite the partial observability and the delay in the common information. 

%Furthermore, the proposed method achieves a sample complexity of [\emph{to be specified}].

We study decentralized partially observable team decision problems with low-rank latent dynamics and unknown system models. The proposed framework combines team-theoretic equivalence with low-rank model representations to address cooperative decision-making in partially observable Markov decision processes without prior knowledge of the transition model. Each team member makes decisions based on local private information and delayed common information shared across the team. Using only this available information, each member learns an approximate low-rank Markov decision process and applies least-squares value iteration to compute its policy. This yields a fully decentralized learning and planning algorithm that requires neither a centralized coordinator nor centralized training. We show that the resulting member-side solutions approximate the centralized team solution: despite partial observability, unknown dynamics, and delayed common information, each member recovers the corresponding component of an approximate team-optimal policy. We further establish finite-sample performance guarantees and derive a corresponding sample-complexity bound for the proposed algorithm.

\end{abstract}


\section{Introduction}
Multi-agent reinforcement learning (MARL) has emerged as an important methodology for sequential decision-making in networked multi-agent systems, with applications in robotic coordination, autonomous driving, wireless networks, and distributed control  \cite{Sumanth2021,sun2025recommendation,zhao2019deep,duan2016benchmarking}. A key difficulty in MARL is coordinating multiple agents that interact with a common environment while making decisions based on decentralized information. Over the past few years, extensive work has been conducted on this problem across various training and execution architectures. Among them, centralized training with decentralized execution (CTDE) has become one of the most widely adopted paradigms, where a centralized module exploits global information during training, while each agent executes its policy using only local observations \cite{amato2024introduction,sunehag2017value,rashid2020weighted,lowe2017multi}. Closely related structures, including centralized evaluation and value decomposition schemes, have also been developed to improve coordination and scalability in cooperative settings. On the algorithmic side, actor-critic, policy-gradient, and value-based methods have all been extended to the multi-agent setting, leading to a broad family of practical MARL algorithms~\cite{hairi2022finite,mao2025decentralized,chen2024decentralized}.

In many applications, however, noisy and limited sensors will prevent the agents from directly observing the underlying system state. This significantly increases the difficulty of decision-making, since the current observation is, in general, insufficient to embed all information relevant for optimal control. 
In order to represent the imperfect observation, such problems are commonly modeled as partially observable Markov decision processes (POMDPs)
% , where the agent must reason over hidden states through its observation--action history 
\cite{aastrom1965optimal,oliehoek2016concise,chen2016pomdp,hauskrecht2000planning}. 
In a POMDP, observation uncertainty induces non-Markovian dependence among successive observations. Therefore, agents need to maintain a belief state, also known as an information state, which summarizes the history by estimating the probability distribution over the latent states and is sufficient for predicting future states and rewards.
However, performing a Bayesian update of the belief requires knowledge of the model, which is not available in MARL settings. More importantly, the belief is a density over the latent state space and is therefore an infinite-dimensional object whenever the state space is continuous. Even in the finite-state case, the set of beliefs that are reachable from distinct histories grows exponentially with the horizon, leading to an intractable representation complexity~\cite{jin2020sample}.

A widely adopted heuristic approach is to extend RL algorithms based on Markov decision process (MDP)
% MDP-based RL algorithms 
by incorporating a sliding window of historical observations to encode the policy or value function, typically implemented with recurrent neural networks~\cite{wierstra2007solving,hausknecht2015deep}.
Several research efforts also construct an approximate information state that compresses the history into a tractable representation for decision-making, and analyze the performance bounds using a corresponding approximate dynamic program with a bounded loss of optimality \cite{subramanian2019approximate,subramanian2022approximate}. 
Another line of work uses representation learning to extract latent features from trajectories or histories to learn a compact representation that is sufficient, or approximately sufficient, to capture the belief \cite{zhang2024provable,uehara2021representation,agarwal2020flambe,efroni2022provable,modi2024model,gao2025spectral,guo2023provably}. 

The challenge increases greatly in the multi-agent reinforcement learning (MARL) setting, where the model, observation kernel, and reward function are unknown. In this case, one must simultaneously address the difficulties caused by partial observation, decentralized information, and unknown environment dynamics. As discussed, existing approaches to multi-agent partial observation typically still rely on some centralized structure, either explicitly or implicitly. In particular, many research efforts on partially observed MARL continue to follow the CTDE paradigm, where a centralized critic, a centralized mixing network, or a centralized state estimator is introduced during training \cite{liu2023partially,liu2026partially,cai2024provable,ni2022representation}. 
Another important line of research in decentralized control and Dec-POMDP is the common-information-based framework, often expressed through prescriptions \cite{subramanian2019approximate,subramanian2022approximate,nayyar2013decentralized,kao2022common}. In this framework, a virtual coordinator observes the common information and selects prescription functions that map each agent's private local information to its local action. This reformulation allows the decentralized control problem to be analyzed as a centralized decision problem from the coordinator's perspective, while preserving decentralized execution at the agent level~\cite{nayyar2013decentralized}.

%
% And for the partial observability case, the approximate information state, which avoids the exact updates of beliefs, has been utilized in \cite{subramanian2019approximate,subramanian2022approximate} to analyze the performance bounds using a corresponding approximate dynamic program with a bounded loss of optimality. 

%Or they rely on some prescription.
%
Although both paradigms can be effective in practice, they share a common limitation: each requires a centralized object (a critic, a mixing network, or a virtual coordinator) to aggregate information that is not locally available to individual agents, and is therefore not fully decentralized. 
This limitation becomes particularly restrictive in large-scale networked systems, or in settings where centralized coordination is costly, unavailable, or incompatible with the system architecture.
Motivated by this gap, we consider a team setting in which all members share the same reward. Team problems provide a natural framework for studying cooperation under decentralized information, since the agents are perfectly aligned through a common objective, and the main challenge arises from the information structure rather than from strategic conflicts~\cite{marschak1955elements,radner1962team,malikopoulos2022team,Xu2025deviation,Xu2025when,Dave2021a}. 
In particular, recent results in team decision theory have shown that, under suitable conditions, one can establish a correspondence between the solution computed by an individual member and the corresponding component of the solution of a centralized manager~\cite{malikopoulos2022team}. 
This viewpoint is appealing because it provides a principled way to reduce a decentralized team problem to member-side decision problems while preserving the team objective~\cite{Malikopoulos2026a}. However, existing results of this type are developed under the assumption that the underlying problem model is known.
Extending the team-theoretic correspondence to a reinforcement learning setting is not straightforward: when the model must itself be learned from data, every member's local decision rule depends on a learned surrogate model, and the global–local equivalence that supports team theory must be re-established under this learned object. 
% This requires the surrogate model to be agent-agnostic and locally reproducible—a property that does not hold for arbitrary nonlinear transition kernels.
 % In particular, this equivalence requires that all members share the same surrogate model, which in turn requires the model class to admit a representation that (i) describes the environment dynamics independently of any individual member, and (ii) can be consistently recovered by each member from the data accessible to the team. 
 
In this paper, we address this gap by studying decentralized team decision problems for a class of $\gamma$-observable POMDPs with unknown model. Our framework combines team-theoretic equivalence with low-rank representation learning. The low-rank structure plays a role specifically tailored to our setting: it characterizes the latent dynamics in an agent-agnostic manner, so that all members can learn the same dynamics representation from delayed common information and then maintain their own beliefs locally with private information. This separation between a globally shared dynamics representation and locally maintained beliefs is what makes fully decentralized learning and execution feasible. Based on the resulting member-side approximate low-rank MDP, each member computes its policy by value iteration without any centralized coordinator at training or execution time. Unlike prescription-based common-information formulations, our methodology does not rely on a virtual coordinator that selects mappings from local information to actions. Instead, we directly construct member-level decision representations and analyze how these local objects collectively approximate the team-optimal behavior under the same surrogate model. As a result, the proposed framework more naturally enables a fully decentralized implementation.
%
% Different from prescription-based common-information formulations, our method does not rely on a virtual coordinator that selects mappings from local information to actions. Instead, we directly construct member-level decision representations and analyze how these local objects collectively approximate the globally desirable team behavior. As a result, the proposed framework more naturally supports fully decentralized implementation.

% The proposed framework connects two lines of research that are typically studied separately. On the one hand, team theory provides a structural interpretation of decentralized optimality by relating member-side and centralized solutions. 
% On the other hand, low-rank modeling and value-iteration-based reinforcement learning offer a tractable route for learning in partially observed en  ·ironments with unknown dynamics. 
% By combining these two viewpoints, we obtain a decentralized approach tailored to partially observed team problems that does not rely on the CTDE architecture. 
% In particular, under suitable structural conditions, we show that the member-side solution yields an approximate team-optimal solution to the original centralized problem, even each member uses only local private data together with delayed common data.

% In this paper, we provide a step toward fully decentralized reinforcement learning for cooperative partially observed systems, and build a bridge between team decision theory and modern low-rank learning methods for structured POMDPs. 
The main contributions of this paper are summarized as follows:
\begin{itemize}
    \item We formulate a decentralized team learning problem for a class
    of $\gamma$-observable POMDPs with an unknown model, in which each member acts based only on its local private information and delayed
    common information, without any centralized coordinator or
    prescription mechanism at training or execution time.

    \item We propose a decentralized representation learning framework
    that exploits the low-rank structure of the latent transition kernel
    as a structural basis for decentralization. The shared dynamics
    representation $(\omega,\psi)$ is learned by all members from delayed
    common information, while member-side representations
    $(\widehat{\phi}^{\,i},\widehat{\mu}^{\,i})$ are formed locally by
    combining this shared component with each member's private belief
    update. Each member then computes its policy via value iteration on
    its own approximate low-rank MDP.

    \item We extend the team-theoretic correspondence between
    member-side and centralized solutions from the known-model setting
    to the model-unknown setting. We establish a global--local
    equivalence (Theorem~\ref{theo:ctde_consistency}) showing that, under the shared surrogate model, each member's locally computed policy coincides with the
    corresponding component of the team-optimal policy under global
    information, and we provide a finite-sample guarantee
    (Theorem~\ref{thm:distributed_team_pac}) on the resulting team performance.
\end{itemize}

The paper is organized as follows. Section II formulates the problem and Section III presents the proposed decentralized algorithm. Section IV provides the convergence analysis and sample-complexity bound. Section V reports numerical experiments and Section VI provides concluding remarks and a glimpse of future research directions.

\section{Preliminaries}

\subsection{Notation}
% For any natural number $n \in \mathbb{N}$, we use $[n]$ to denote the set
% $\{1,\cdots,n\}$. 
For vectors we use $\|\cdot\|_p$ to denote the
$\ell_p$-norm, and we use $\|x\|_{\Lambda}$ to denote $\sqrt{x^\top
\Lambda x}$.
For two probability distributions $\nu$ and $\nu'$ over a measurable
space $\mathcal X$, we use $ \|\nu-\nu'\|_{\mathrm{TV}} $
to denote their total variation distance, defined by $ \|\nu-\nu'\|_{\mathrm{TV}} =
\sup_{A\subseteq \mathcal X} |\nu(A)-\nu'(A)|. $
% When $\mathcal X$ is finite or countable, this is equivalent to
% \[
% \|\nu-\nu'\|_{\mathrm{TV}}
% =
% \frac{1}{2}\sum_{x\in\mathcal X} |\nu(x)-\nu'(x)|.
% \]
For two sequences $\{a_n\}_{n\ge 1}$ and $\{b_n\}_{n\ge 1}$, we use $a_n \le O(b_n)$ to denote that there exists a constant $C>0$ such that $a_n \le C b_n$.
For any natural number $n \in \mathbb{N}$, we use $[n]$ to denote the set $\{1,...,n\}$.
For a set $S$, we use $\Delta(S)$ to denote the set of all
probability distributions on $S$. For an operator $O:S\to \mathbb{R}$ and
$b\in \Delta(S)$, we use $Ob:\mathcal{O}\to \mathbb{R} $
to denote $\int_S O(o\mid s)b(s)\,ds. $
We consider a finite-horizon partially observable Markov
decision process (POMDP) $\mathcal{P}$, which can be specified as a tuple \[ \mathcal{P} = (\mathcal{S},\mathcal{A}, H, \mathcal{O}, d_0, \{r_h\}_{h=1}^H, {P}, O), \]
where $\mathcal{S}$ is the state space, $\mathcal{A}$ is a finite set of
actions, $H\in \mathbb{N}$ is the episode length, $\mathcal{O}$ is the set
of observations, and $d_0$ is the known initial distribution over states.  
We use ${P}:\mathcal{S}\times\mathcal{A}\to \Delta(\mathcal{S})$
to denote the transition kernel and $r_h:\mathcal{O}\to \mathbb{R}$ is the reward function at step $h$.
We use $O:\mathcal{S}\to\Delta(\mathcal{O})$ to denote observation kernel, where for $s\in\mathcal{S}$ and $o\in\mathcal{O}$, $O(o\mid s)$ is the probability of observing $o$ while in state $s$. 
% We denote
% \[
% r=(r_1,\cdots,r_H),\qquad
% \mathbb{P}=(\mathbb{P}_1,\cdots,\mathbb{P}_H),
% \qquad
% O=(O_1,\cdots,O_H).
% \]
Note that the MDP is a special case of POMDP, where $\mathcal{O}=\mathcal{S}$ and $ O(o\mid s)=\mathbf{1}\{o=s\} $ for all $h\in[H]$, $o\in\mathcal{O}$ and $s\in\mathcal{S}$.

\subsection{Problem Formulation}
We consider a partially observable team with $N$ members, represented by the tuple
\[ \mathcal{G}=(H,\calS,\{\calA^i\}_{i=1}^N,\{O^i\}_{i=1}^N, {P},d_0, r),
\]
where $H\in\Nat$ is the episode length, $\calS$ is the discrete state space (we write $|\calS|=S$), and for member $i$, $\calA^i$ and $\calO^i$ are its discrete action and observation spaces (with $|\calA^i|=A^i$ and $|\calO^i|=O^i$). 

Denote $a=(a_1,\dots,a_N)\in\calA, \calA=\calA^1\times\cdots\times\calA^N,  |\calA|=\prod_{i=1}^N A^i $.
% \begin{equation}
% a=(a_1,\dots,a_N)\in\calA, \calA=\calA^1\times\cdots\times\calA^N,  |\calA|=\prod_{i=1}^N A^i\cdot
% \end{equation}
% \begin{equation}
% P(\cdot\,|\,s,  a)\in\calP(\calS),
% \end{equation}
The transition kernel $P(\cdot\,|\,s,  a)\in\calP(\calS)$ gives the distribution of the next state given $(s, a)$, where $\calP(\calS)$ denotes the set of all probability distributions on $\calS$.
% For convenience, we may view, for a fixed $  a$, the matrix
% \[
% P_h(  a)\in\mathbb{R}^{S\times S},
% \]
% whose $s$-th row is the next-state distribution from state $s$.
The initial state is denoted as $d_0$ and it satisfies $d_0\sim\mu_0$ with $\mu_0\in\calP(\calS)$.

We write $O^i(\cdot\,|\,s)\in\calP(O^i)$ for the $i$-th member’s observation kernel, which denotes the observation distribution at state $s$. Let $O = (O^1,\dots,O^N)$ be the joint observation kernels of all members and
\[  O=O^1\times\cdots\times O^N,\quad  O=\prod_{i=1}^N O^i. 
\]
% The collection of observation kernels is $\mathbb{O}=\{\calO\}_{h = 1,...,H}$, where $  \calO(\cdot\,|\,s)\in\calP(O)$
% is the joint observation distribution at state $s$ and step $h$.
%
The reward $r_{h}(o_h): \mathcal{O} \rightarrow \mathbb{R}$ is the reward function given the joint observation $o_h \in \mathcal{O}$. 
% We denote $ r = (r_1,..., r_N)$.

%\paragraph{Team dynamics}
Next, we describe the partially observable stochastic dynamic team.
At the beginning of each episode, $s_0\sim d_0$. At step $h\in[H]$, each member $i$ receives a private observation $o^i_{h}$; the joint observation $  o_h=(o^1_{h},\dots,o^N_{h})$ is drawn from $O(\cdot\,|\,s_h)$ and rewards $r^i_{h}(o_h)$ are realized. 
Then each member takes an action $a^i_{h}$ and the environment transitions to
\[
s_{h+1}\sim P(\cdot\,|\,s_h,  a_h), \quad   a_h=(a^1_{h},\dots,a^N_{h}).
\]
The episode terminates after step $H \in \mathbb{N}$, a terminal observation $o_{H+1}\sim \calO_{H+1}(\cdot\,|\,s_{H+1})$ is drawn and terminal rewards $r^i_{H+1}(  o_{H+1})$ are received.
This cooperative case is also known as a decentralized partially observable Markov decision process (Dec\mbox{-}POMDP).
We define the value function for policy $\pi$ at step $h$ by
\[ V_h^\pi = \mathbb{E}^{P}_{a_{h}\sim \pi}
\left[ \sum_{t=h}^{H}r_t(o_t) \right],
\]
namely as the expected reward received by following $\pi$.

In the following, we move to the information sharing structure in the team.
% \subsubsection{Information sharing in team}
%
Each team member $i$ maintains its own history information $\tau^i_{h}$,
which consists of all its historical observations and actions up to step~$h$:
\begin{equation} \label{eq:tau_team}
\tau^i_{h} = \{ a_1, o_2, \ldots, a_{h-1}, o_h\}.
\end{equation}
We denote by $\mathbf{\tau}_h = \{\tau^i_{h}\}_{i=1}^N$ the collection of all members' histories at step~$h$.
In many practical scenarios, members may share part of their information with other members.
% , thereby inducing a structured information pattern that can improve sample and computational efficiency.
The shared part is referred to as \emph{common information}, while the remaining part for each member is its \emph{private information}.

% \paragraph{Common and private information.}
Let $c_h \subseteq \eta_h = \{a_1,o_1,\ldots,a_h,o_h\}$ denote the common information available to all members at step~$h$, and let $\mathcal{C}_h$ denote the collection of all such common information realizations.
Given $c_h$, the private information of member~$i$ is defined as
\begin{equation} \label{eq:private information}
p^i_{h} = \tau^i_{h} \setminus c_h,
\end{equation}
and the collection of private information across all members is denoted by
$p_h = \{p^i_{h}\}_{i=1}^N$.

We consider that both common and private information are updated over time according to recursive rules.
The common information $c_h$ is non-decreasing with~$h$, that is,
$c_h \subseteq c_{h+1}$ for all~$h$.
Let ${\kappa}_{h+1} = c_{h+1} \setminus c_h$ denote the new information revealed at step~$h+1$.
Then the evolution of common information can be represented as
\begin{equation} \label{eq:common_info_update}
c_{h+1} = \{c_h, \kappa_{h+1}\}, 
\kappa_{h+1} = \chi_{h}(p_h, a_h, o_{h+1}),
\end{equation}
where $\chi_{h}$ is a fixed transformation.
Similarly, the private information of member~$i$ evolves according to
\begin{equation}  \label{eq:private_info_update}
p^i_{h+1} = \nu^i_{h}(p^i_{h}, a^i_{h}, o^i_{h+1}),
\end{equation}
where $\nu^i_{h}$ is also a fixed transformation.
% We use $z_{h+1}$ to denote the collection of $\{z_{i,h+1}\}_{i=1}^N$ at step~$h{+}1$.
% Since the trajectory typically starts from $a_1$ instead of $o_1$, we set $c_1 = \emptyset$.

In this paper, we adopt an \emph{$n$-step delayed sharing information structure}~\cite{nayyar2010optimal}. The common information $\Delta_t$ at time $t$ includes joint observations and actions from $n$ steps prior, while private information $\Lambda_t^i$ contains local data from the recent $n$ steps.
%their
Thus, each member can only use their delayed common information and the private information to execute the actions during the decentralized algorithm.

% Thus, at each step $h$, member $i$ selects its action
% based only on the delayed common information $c_h$ and its private
% information $p_h^i$.

{
The objective of the partially observable team is to find a joint policy that maximizes the expected cumulative team reward under the above information-sharing structure. This problem can be formalized as follows:
\begin{equation}
    \pi^\star
    \in
    \arg\max_{\pi\in\Pi_{\mathrm{dec}}}
    V_1^\pi,
    \label{eq:team_problem}
\end{equation}
where $\Pi_{\mathrm{dec}}$ denotes the set of admissible joint policies
$\pi=(\pi^1,\ldots,\pi^N)$ satisfying
$\pi_h^i(\cdot\mid c_h,p_h^i)\in\Delta(\mathcal{A}^i)$ for all
$i\in[N]$ and $h\in[H]$. 
}

{
\begin{remark}
Although the information structure considered in this paper allows members to receive delayed common information generated by other members, we use the term "decentralized" following the convention in the common-information literature~\cite{nayyar2013decentralized,kao2022common,malikopoulos2022team}.
The problem remains decentralized in the sense that each member selects its action based only on its own information set, consisting of its private information and the available delayed common information, without a centralized decision maker determining the joint action.
Thus, while the implementation may also be viewed as
distributed due to information exchange, the underlying decision problem is referred to here as decentralized.
\end{remark}
}

\subsection{Representation Learning for Low-rank POMDP}
We consider the case where the transition kernel of the environment admits a low-rank structure. This assumption plays a structural role specifically tailored to our decentralized setting. 
Specifically, all members can learn the same representation from delayed common information and then maintain their own beliefs locally using private information. This {\it separation} between a shared dynamics representation and locally updated beliefs makes fully decentralized learning and execution feasible and supports the global–local equivalence established in Theorem~\ref{theo:ctde_consistency}.
We formalize the low-rank structure as follows.

\begin{definition}(Low-rank transition) \label{defi:low-rank}
A transition kernel ${P}: \mathcal{S} \times \mathcal{A} \rightarrow \Delta (\mathcal{S})$ admits a low-rank decomposition of dimension $d$ if there exist two mappings $\omega^*: \mathcal{S} \rightarrow \mathbb{R}^d$ and $\psi^* : \mathcal{S} \times \mathcal{A} \rightarrow \mathbb{R}^d$ such that 
\begin{equation}
{P}(s' | s, a) = \omega^*(s')^\top \psi^*(s, a).
\end{equation}
\end{definition}

The mappings $\omega^*$ and $\psi^*$ describe latent state-action dynamics, once $(\omega^*, \psi^*)$ is learned from data accessible to the team, each member can construct its own member-side representation by combining these shared mappings with its locally available beliefs. We further focus on the observability POMDP setting, formalized below.

% Also, we focus on the observability POMDP setting, as shown in the following assumption.
\begin{assumption}\label{ass:observability}
For $i \in [N]$, let $\gamma^i> 0$, $O$ be the operator with $O(\cdot \mid s)$, indexed by states $s$. The operator $O$ satisfies $\gamma$-observability for any distributions $b, b'$ over states, i.e.,
\begin{equation}
\|O^i b - O^i b'\|_1 \geq \gamma^i \|b - b'\|_1.
\end{equation}
% A POMDP satisfies $\gamma$-observability if all $h \in [H]$ satisfy $\gamma$-observability.
\end{assumption}

Assumption~\ref{ass:observability} implies that the members' operator is an injection. 
We also make the following assumptions for the joint observation of the team.
\begin{assumption}\label{ass:factorized_joint_gamma_obs}
The joint-observation model satisfies the same observability condition as in the original low-rank framework, with the single observation space replaced by the joint observation space $ \mathcal O^{\mathrm{joint}}=\mathcal O^1\times\cdots\times\mathcal O^N.$
Specifically, the corresponding observability operator associated with the joint decoder
\[
O^{\mathrm{joint}}(o^{1:N}\mid s) = \prod_{i=1}^N O^{i}(o^i\mid s),
\]
is assumed to be nondegenerate, with observability constant $\gamma > 0$.
\end{assumption}

\begin{assumption}
\label{ass:factorized_joint_obs}
For each stage $h\in[H]$, let
\begin{equation}
o_h^{1:N} = (o_h^1,\dots,o_h^N)\in O^1\times\cdots\times O^N,
\end{equation} denote the collection of observations received from all members. Conditioned on the latent state $s_h\in\mathcal S$, the members' observations are conditionally independent, i.e.,
\begin{equation}
O^{\mathrm{joint}}(o_h^{1:N}\mid s_h)
=\prod_{i=1}^N O^i(o_h^i\mid s_h),
\end{equation}
where $O^i(\cdot\mid s_h)$ is the local observation kernel of member $i$.
\end{assumption}

\begin{remark}
In a decentralized team setting, the manager aggregates local observations from all members and treats them as a joint observation variable. 
Under the conditional independence assumption across members, the resulting joint observation kernel admits a factorized form.  
Consequently, the use of joint observations modifies the observation decoder but does not alter the latent predictive representation, provided that the induced joint-observation model still satisfies a $d$-dimensional realizability condition and a $\gamma$-observability condition.
\end{remark}

It has been shown that when Assumption~\ref{ass:observability} holds, the POMDP can be approximated by MDP with state space $\mathcal{Z} = {\calO}^L \times \mathcal{A}^{L-1}$ (see \cite{guo_randcom_2023,zhang2024provable,efroni2022provable,uehara2022provably}).
In our work,
we also consider $L$-memory policies. 
For all $h \in [H]$, let $\mathcal{Z}_h = {\calO}^L \times \mathcal{A}^{L-1}$.
An element $z_h \in \mathcal{Z}_h$ is represented as 
\begin{equation}
z_h = (o_{h+1-L:h},\, a_{h+1-L:h-1}),
\end{equation}
where $o_{h+1-L:h} = (o_{h+1-L}, \ldots, o_h)$ 
and $a_{h+1-L:h-1} = (a_{h+1-L}, \ldots, a_{h-1})$.
We also make a commonly used realizability assumption that the given function class contains the true function.
\begin{assumption} \label{ass:model_class}
There exists a known model class
\begin{equation}
\mathcal{F} =
\left\{
\left( \omega, \psi \right) : \omega \in \Omega, \psi \in \Psi
\right\},
\end{equation}
such that \(\omega^* \in \Omega\) and \(\psi^* \in \Psi\).
Recall that according to Definition~\ref{defi:low-rank}, the new observation conditioned on the state-action pair satisfies
\begin{align}
&P(o_{h+1} \mid s_h, a_h) \nonumber
\\&=O_{h+1}(o_{h+1} \mid s_{h+1})
\,{\omega}(s_{h+1})^\top
\psi(s_h, a_h).
\end{align}
\end{assumption}

% With Assumption~\ref{ass:model_class}, for any 
% $z_h = (o_{h+1-L:h},\, a_{h+1-L:h-1})$, $a_h$, and $o_{h+1}$, we have
% \begin{align}
% &{P}(o_{h+1} \mid o_{h+1-L:h}, a_{h+1-L:h}) \nonumber
% \\&= \Biggl[\int_{\mathcal{S}'} 
%     \omega_h^{\star}(s')^{\top} 
%     O_{h+1}^{\star}(o_{h+1} \mid s') 
%     \, \mathrm{d}s' \Biggr]
%     \phi^{\star}(z_h,a_h).
% \end{align}

For any $o_h \in \mathcal{O}$, we denote
\begin{equation}\label{eq:mu}
\mu(o_h) = \int_{\mathcal{S}'} 
\omega(s')^{\top} O(o_h \mid s')\mathrm{d}s'.
\end{equation}

% According to \cite{uehara2022provably},
Recall that when Assumption~\ref{ass:observability} holds, the POMDP can be approximated by an MDP whose state space is $\mathcal{Z}_h = {O}^L \times \mathcal{A}^{L-1}$.
Specifically, for any $\mathcal{P} = (O, \omega, \psi)$, we can construct an approximated MDP
\begin{equation} 
\mathcal{M} = (\mu,\phi),
\end{equation}
where \((\phi,\mu)=q(\omega,\psi)\) for an explicit function \(q\), which is defined in detail in Section~\ref{sec:algorithm_design}.
% can be obtained by \eqref{eq:mu} and \eqref{eq:phi}.
This approximated MDP \(\mathcal{M}\) satisfies that
\begin{equation} \label{eq:MDP}
P^{\mathcal{M}}(o_{h+1}\mid z_h,a_h)
= \mu^{\top}(o_{h+1})\,\phi_h(z_h,a_h).
\end{equation}
At the same time, the POMDP \(\mathcal{P}\) satisfies that
\begin{equation} \label{eq:pomdp}
{P}^{\mathcal{P}}(o_{h+1}\mid \eta_h,a_h)
= \mu^{\top}(o_{h+1})\,\xi_h({\eta}_h,a_h),
\end{equation}
where the definition of $\xi_h$ can be found in~\eqref{eq:xi} in Section~\ref{sec:algorithm_design}, $\eta_h = (o_{3-2L:h}, a_{3-2L:h})$.
Following
\cite{guo2023provably}, we consider an extended POMDP, where dummy observations
and actions before the first real decision stage are introduced only for
notational convenience. This construction allows the finite-memory window to
be written uniformly for all stages, including the early stages. These dummy
variables do not affect the initial state distribution or the actual
interaction with the environment.

\section{Algorithm Design} \label{sec:algorithm_design}
% \subsection{Algorithm Design}
In this section, we present a decentralized algorithm for solving the team decision problem with an unknown model. The design reflects the {\it structural separation}: the latent dynamics can be learned from data shared across the team, whereas beliefs depend on each member's own observation history and must be maintained locally.
The implementation is divided into three steps, and the methodological procedure  is given in Algorithm~\ref{alg:ma-porl-jr}.

\paragraph{Step 1: Decentralized low-rank learning of transition matrix}
Using the delayed common information available to the entire team, each member performs the same maximum-likelihood estimation from the delayed common buffer to learn a shared representation $\omega$ and $\psi$ of the low-rank transition kernel $P$, together with a common exploration bonus $\hat{b}$. The features $\omega$ and $\psi$ describe only the latent state-action dynamics, the resulting estimate is identical across members and does not require any centralized coordinator at execution time.

\paragraph{Step 2:  Decentralized local information state update}
Each member $i$ then combines the shared dynamics representation
$(\omega,\psi)$ with its own private information to construct
member-side representations $(\widehat{\phi}^{\,i},\widehat{\mu}^{\,i})$.
The transition component is inherited from the team-level estimation,
while the belief update is carried out locally through the operator
$q^i$. This is precisely the shared-dynamics/local-belief decomposition enabled by the low-rank structure.

\paragraph{Step 3: Decentralized Execution}
Each member $i$ uses its local representations
$(\widehat{\phi}^{\,i},\widehat{\mu}^{\,i})$ and the exploration bonus
to compute its policy via a finite-horizon backward value iteration, 
{
which recursively evaluates the local value function and selects the greedy action at each stage,
}
as described in 
Algorithm~\ref{alg:ma-lsvi-jr}. No inter-member communication or centralized prescription
is required at execution time. The global-local equivalence established
in Theorem~\ref{theo:ctde_consistency} ensures that the resulting member-side policies jointly approximate the team-optimal solution under the same surrogate model.

% ===================== Algorithm: Joint-policy rollout + local learning =====================
\begin{algorithm}[!t]
\caption{Decentralized Partially Observable Team with non-classical information sharing}
\label{alg:ma-porl-jr}
\textbf{Require:} members $i\in [N]$, horizon $H$, memory $L$, representation classes $\{\mathcal F_h^i\}_{h=1}^{H}$, outer iterations $K$, parameters $\{\alpha_k,\lambda_k\}$.

\begin{algorithmic}[1]
\State Initialize per-member policies $\{\pi^{i}_0\}_{i=1}^N$, buffers $\mathcal D_h^i\gets\varnothing$ for all $h,i$.
\For{$k=1,2,\ldots,K$}
\CommentState{Sampling under joint-policy}

\For{each $h \in [H]$}
 \CommentState{Round-1 (cut at $h{-}L$):} 
run the joint policy $\pi^{k-1}$ from episode start to $h{-}L$; from $h{-}L$ to the end use uniform policy $U(\mathcal{A}^i)$. 
Local observation and action: $y^i_h = (o^i_{h-L+1:h}, a^i_{h-L+1:h-1})$. Append $(y_h^i,a_h^i,o_{h+1}^i)$ to $\mathcal D_h^i$.
   \CommentState{Round-2 (cut at $h{-}2L$):}
run $\pi^{k-1}$ to $h{-}2L$; from $h{-}2L$ to end use uniform policy $U(\mathcal{A}^i)$. Append $(y_h^i,a_h^i,o_{h+1}^i)$ to $\mathcal D_h^{i'}$.
\EndFor

\CommentState{Decentralized representation learning}

  \State The members receive the joint information $(y_h,\,a_h,\,o_{h+1})$  until $H-n$, i.e., 
  $\mathcal D_h = \mathcal D^1_h \cup \cdots \cup  \mathcal D^N_h$, $\mathcal D'_h = \mathcal D^{1'}_h \cup \cdots \cup  \mathcal D^{N'}_h$.

  \State
Members learn representations of the system dynamics $\hat{\omega}_{k}, \hat{\psi}_{k}$ with delayed common information until $H-n$: $\mathcal D = \mathcal D_1 \cup ... \cup \mathcal D_{H-n}$ and $\mathcal D' = \mathcal D'_1 \cup ... \cup \mathcal D'_{H-n}$ according to \eqref{eq:representation_learning}.
% \begin{align*}
% &(O_{k}, \omega_{k}, \psi_{k}) \\&=\arg\max_{(O,\omega,\psi)\in\mathcal{F}}
% \mathbb{E}_{D\cup D^{'}}\!\left[
% \log \xi(\eta_h,a_h)^\top \mu(o_{h+1})
% \right].
% \end{align*}
% according to \eqref{eq:xi}.

\State
Members learn the common system model feature $\widehat{\phi}_{k,h} $, $\widehat{\mu}_{k}$ according to~\eqref{eq:phi} and~\eqref{eq:mu} for $h = 1,...,H-n$, and the corresponding exploration bonus: 
\begin{equation}
\widehat{b}_{k,h}(z,a) = \min\left\{
\alpha_k \sqrt{\widehat{\phi}_{k,h}(z,a)^\top \Sigma^{-1}\widehat{\phi}_{k,h}(z,a)}, \, 1 \right\},
\end{equation} 
where $\Sigma = \sum_{z\sim \mathcal{D} \cup \mathcal{D'}} \widehat{\phi}_{k}(z,a)\widehat{\phi}_{k}(z,a)^\top + \lambda_k \mathrm{I}$.
And for $h = H-n+1,...,H$, set $\hat{b}_{k,h} = \hat{b}_{k,H-n}$.
 
\State
Similarly, member $i$ learns its own features $\widehat{\phi}^{i}_{k,h}$ and $\widehat{\mu}^{i}_{k}$ using its private information $p^i_{h}$ and the delayed common information $c_h$, with the common approximate system dynamics  $\hat{\omega}_{k}, \hat{\psi}_{k}$, $h=1,...,H$ according to \eqref{eq:local-q-map}:
\begin{equation}
(\widehat{\phi}^{i}_{k,h}, \widehat{\mu}^{i}_{k}) = q^i(\hat{\omega}_{k}, \hat{\psi}_{k})
\end{equation}
\CommentState{Decentralized Local policy update:} 

\State
Call Algorithm~\ref{alg:ma-lsvi-jr} with the exploration bonus $\tilde{b}_{k} = \operatorname{col}\{\hat{b}_{k,h}\}_{h= 1,...,H}$
% \begin{equation}
% \tilde{b}_{k} = \operatorname{col}\{\hat{b}_{k,h}\}_{h= 0,...,H-1}\end{equation}
and the local representation functions,  $\{\hat{\phi}^{i}_{k,h}\}_{h= 1,...,H}, \hat{\mu}^{i}_{k}$ to obtain local policy $\pi^{i}_{k}$.

\EndFor
\State \textbf{return} the decentralized policies 
$\{\pi^{i}_K\}_{i=1}^{N}$.
\end{algorithmic}
\end{algorithm}

% ===================== Per-agent LSVI (unchanged structure, uses only local info) joint action =====================
% \begin{algorithm}[!t]
% \caption{Value Iteration (Agent $i$)}
% \label{alg:ma-lsvi-jr}
% \textbf{Require:} rewards $\{r_h\}$, bonus $\{b^i_h\}$, features $\phi^i_h$, $\mu^i$, $h= 1,...,H$.

% \begin{algorithmic}[1]
% \State Initialize $V_H^i(z)=0$ for any local info-state $z$.
% \For{$h=H-1$ \textbf{down to} $1$}
%   \For{$(z_h^i,a_h)\in \mathcal Z^i\times \mathcal A$}
% \begin{equation}
% \begin{aligned}
% & Q_h^i(z_h^i,a_h) = r_h+b_h \\& +\sum_{o_{h+1}^i} \big(\phi^i_h(z_h^i,a_h)\big)^\top \mu^i(o_{h+1}^i)
% V_{h+1}^i(z^i_{h+1})
% \end{aligned}
% \end{equation}
% where $z^i_{h+1} = (o_{h-L+2:h+1}^i, a_{h-L+2:h}^i )$.

% \EndFor
% \State 
% $V_h^i(z) \leftarrow  \max_{a\in \mathcal A} Q_h^i(z,a)$,

% $a  \leftarrow \arg\max_{a\in \mathcal A} Q_h^i(z,a) $
% \State
% % $\pi_h^i(z) \leftarrow \arg\max_{a\in \mathcal A^i} Q_h^i(z,a)$
% $\pi_h^i(z) \leftarrow a^i$
% \EndFor
% \State \textbf{return} $\pi^i=\{\pi_h^i\}_{h=1}^{H}$.
% \end{algorithmic}
% \end{algorithm}

% ===================== Per-agent LSVI (unchanged structure, uses only local info) =====================
\begin{algorithm}[!t]
\caption{Value Iteration (Member $i$)}
\label{alg:ma-lsvi-jr}
\textbf{Require:} rewards $\{r_h\}$, bonus $\{b^i_h\}$, features $\phi^i_h$, $\mu^i$, $h= 1,...,H$.

\begin{algorithmic}[1]
\State Initialize $V_H^i(z)=0$ for any local info-state $z$.
\For{$h=H-1$ \textbf{down to} $1$}
  \For{$(z_h^i,a_h^i)\in \mathcal Z_h^i\times \mathcal A^i$}
\begin{equation}
\begin{aligned}
& Q_h^i(z_h^i,a_h^i) = r_h+b_h \\& +\sum_{o_{h+1}^i} \big(\phi^i_h(z_h^i,a_h^i)\big)^\top \mu^i(o_{h+1}^i)
V_{h+1}^i(z^i_{h+1})
\end{aligned}
\end{equation}
where $z^i_{h+1} = (o_{h-L+2:h+1}^i, a_{h-L+2:h}^i )$.

\EndFor
\State 
$V_h^i(z) \leftarrow  \max_{a\in A^i} Q_h^i(z,a)$,
\State
$\pi_h^i(z) \leftarrow \arg\max_{a\in A^i} Q_h^i(z,a)$
\EndFor
\State \textbf{return} $\pi^i=\{\pi_h^i\}_{h=1}^{H}$.
\end{algorithmic}
\end{algorithm}

We would like to remark that the two exploratory datasets $\mathcal D$ and $\mathcal D'$ in Algorithm~\ref{alg:ma-porl-jr} play different roles: the first one provides coverage for the $L$-memory information state at step $h$, 
while the second one, generated with uniform exploration over the last $2L$ steps, provides additional coverage after moving the analysis
back by $L$ steps, 
% while the second one is used to support the $L$-step distribution-backtracking argument in the analysis, 
following the exploration design in~\cite{guo2023provably}.
{
Besides, the shared dataset $\mathcal D_h=\cup_i \mathcal D_h^i$ should not be interpreted as a centrally stored training set. Rather, it represents the delayed common information that becomes available to all members under the assumed information-sharing structure. 
% Hence, each member can reconstruct the same shared surrogate model from the same delayed common information, while its belief update and policy computation remain local.
}

In the following, we introduce the detailed construction of the representations.
We first calculate the belief, i.e., the information state, which is the conditional probability of state $s_h$ given the true transition and an action and observation sequence $\{o_{3-2L}, a_{3-2L}, \cdots, a_h, o_h\}$ and $1 \le h \le H$ respectively.
Consider a POMDP and a history $(o_{3-2L:h}, a_{3-2L:h-1})$, the belief
$b_h^{\mathcal P}(o_{3-2L:h}, a_{3-2L:h-1}) \in \Delta(\mathcal S)$ is given by the distribution of the state $s_h$ conditioned on taking actions $a_{1:h-1}$ and observing $o_{1:h}$ in the first $h$ steps. Formally, the belief state is defined inductively as follows:
\begin{equation}
b_1^{\mathcal P}(\varnothing)=b_1,
\end{equation}
where $b_1$ is a properly chosen prior distribution whose precise form is deferred to \cite{guo2023provably}.
For $2 \le h \le H$ and any $(o_{3-2L:h}, a_{3-2L:h-1}) \in \mathcal{H}_h$, define
\begin{align}
&b_h^{\mathcal P}(o_{3-2L:h}, a_{3-2L:h-1})
\nonumber
\\&=U_{h-1}^{\mathcal P}\bigl(
b_{h-1}^{\mathcal P}(a_{1:h-2},o_{2:h-1});
\,a_{h-1},o_h
\bigr),
\end{align}
where for $b\in\Delta(\mathcal S)$, $a\in\mathcal A$, $o \in\mathcal O$, the belief update operator
$U^{\mathcal P}$ is defined as
\begin{equation}
U^{\mathcal P}(b;a,o)(s) =
\frac{
O_{h+1}(o\mid s)\cdot \sum_{s'\in\mathcal S} b(s')\cdot P(s\mid s',a)
}{\sum_{x\in\mathcal S} O_{h+1}(o\mid x)\sum_{s'\in\mathcal S} b(s')\cdot  P(x\mid s',a)
}.
\end{equation}

% {\color{red} What is the new observation for the members, does it include the other members delayed information, should it only use its own observation function, if only use its own observation function, or it should include the other members delayed information, and other members observation function? leave it as it is now
% }

This operator calculates the belief state for the $h+1$-step when the belief for the $h$-step is $b$, and after the member takes action $a$ and receives the observation $o$.
Recall that the transition matrix has a low-rank structure, we have
\begin{align}
& P^{\mathcal P}(o_{h+1}\mid o_{3-2L:h},a_{3-2L:h})
\nonumber
\\&= \int_{\mathcal S_{h+1}} \omega(s_{h+1})\cdot O(o_{h+1}\mid s')
\,ds'\nonumber
\\& \cdot \int_{\mathcal S} \psi(s,a_h)\,
b_h^{\mathcal P}(o_{3-2L:h},a_{3-2L:h-1})(s)\,ds .
\end{align}
For $h\in[H]$, we denote
\begin{equation} \label{eq:xi}
\xi_h(\eta_h,a_h) = \int
\psi(s_h,a_h)\,
b_h^{\mathcal P}(\eta_h)(s_h)\,d s_h,
\end{equation}
we define the approximated belief $\bar b_h(o_{h-L:h},a_{h-L:h-1})$ to approximate the true belief $b_h(o_{3-2L:h},a_{3-2L:h-1})$.

The Bellman recursion is evaluated with respect to the member-side predictive distribution induced by the common surrogate model. By Theorem 1, this recursion is equivalent to the corresponding component of the manager-side Bellman recursion.

For $b\in\Delta(\mathcal S)$ and $o\in\mathcal O$, define $B(b,o)$ as the operation that incorporates observation $o$ by
\[ B(b,o)(s)=\frac{
O(o\mid s)\cdot b(s)
}{ \sum_{x\in\mathcal S} O(o\mid x)\sum_{s'} b(s') },
\]
which denotes the belief distribution after receiving the observation $o$ as the original belief distribution was $b$.
For an action and observation sequence $\{o_{3-2L}, a_{3-2L}, \cdots, o_H, a_H\}$ and $2\le h \le H$. We define the approximated belief as:
\begin{align} \label{eq:finite-memory-belief-recursion}
&\bar b_{h-L} = B(\tilde b_0^{\,h-L},\, o_{h-L}),
\nonumber\\&
\bar b_{h-L+\ell}(o_{h-L:h-L+\ell},a_{h-L:h-L-1+\ell})
\nonumber\\&=
U_{h-L-1+\ell}^{\mathcal P}
\Bigl(
\bar b_{h-L-1+\ell}(o_{h-L:h-L-1+\ell},a_{h-L:h-L-\ell}),
\nonumber\\&o_{h-L+\ell},a_{h-L-1+\ell}
\Bigr),
\quad 1\le \ell \le L.
\end{align}
% where the construction of the initial belief $\tilde b_0^{\,h-L}$ can be found in \cite{guo2023provably}.

The resulting belief at the end of this recursion is denoted by
\(\bar b_h(z_h)\). In other words, \(\bar b_h(z_h)\) is obtained by applying
Bayesian filtering only over the recent finite-memory window instead of the
entire history \(\eta_h\).

Replacing the true belief in \eqref{eq:xi} with the approximate belief gives
the finite-memory feature
\begin{equation}
\phi_h(z_h,a_h)
=
\int_{\mathcal S}
\psi(s_h,a_h)
\bar b_h(z_h)(s_h)\,ds_h .
\label{eq:phi}
\end{equation}
Consequently, the original POMDP induces an approximate low-rank MDP
\(\mathcal M\) over the finite-memory state \(z_h\), whose transition kernel is
given by
\[
P^{\mathcal M}(o_{h+1}\mid z_h,a_h)
=
\mu(o_{h+1})^\top \phi_h(z_h,a_h).
\]
Thus, planning in the original POMDP can be approximated by planning in an MDP
whose state is the finite-memory representation \(z_h\) and whose transition kernel has a low-rank form.

% Initialize policy $\pi_0 = \{\pi_{0,1}, \ldots, \pi_{0,H}\}$ to be arbitrary policies  and replay buffers $\mathcal{D}_h = \emptyset$, $\mathcal{D}'_h = \emptyset$ for all $h$. 
% %
% Data collection from $\pi_{k-1}$, $\forall h \in [H]$, 
% $\eta_h^{k-1} \sim d^{\pi_{k-1}}_{h, \mathcal{O}_1, \mathcal{U}(\mathcal{A})}$, 
% $\mathcal{D}_h = \mathcal{D}_h \cup \{\eta_h\}$, 
% $\tilde{\eta}_h^{k-1} \sim d^{\pi_{k-1}}_{h, \mathcal{O}_2, \mathcal{U}(\mathcal{A})}$, 
% $\mathcal{D}'_h = \mathcal{D}'_h \cup \{\tilde{\eta}_h\}$.
%
%
%

At iteration \(k\), for each \(h\in[H]\), we collect data using the policy
\(\pi_{k-1}\) with additional uniform exploration. Specifically, we sample
\[
\eta_h^{k}
\sim
d^{\pi_{k-1}\circ_L \mathcal U(\mathcal A)}_{h},
\qquad
\tilde \eta_h^{k}
\sim
d^{\pi_{k-1}\circ_{2L} \mathcal U(\mathcal A)}_{h},
\]
and update the replay buffers as
\[
\mathcal D_h
\leftarrow
\mathcal D_h\cup\{\eta_h^{k}\},
\qquad
\mathcal D'_h
\leftarrow
\mathcal D'_h\cup\{\tilde \eta_h^{k}\}.
\]
Here \(d^{\pi_{k-1}\circ_L \mathcal U(\mathcal A)}_{h}\) denotes the
distribution of the stage-\(h\) history induced by following \(\pi_{k-1}\)
and applying the uniform exploration policy \(\mathcal U(\mathcal A)\) over
the most recent \(L\) steps. Similarly,
\(d^{\pi_{k-1}\circ_{2L} \mathcal U(\mathcal A)}_{h}\) uses uniform
exploration over the most recent \(2L\) steps.
Then we learn and update the representation $\psi_k$ and $\omega_k$ for the low-rank transition matrix for the latent system, as follows
\begin{align} \label{eq:representation_learning}
&(\omega_{k}, \psi_{k}) 
\nonumber
\\&=\arg\max_{( \omega, \psi)\in\mathcal{F}}
\mathbb{E}_{D_h\cup D^{'}_h} \left[
\log \xi_h(\eta_h,a_h)^\top \mu(o_{h+1})
\right],
\end{align}
where $\mu$ is computed by~\eqref{eq:mu}, $\xi$ is defined in~\eqref{eq:xi}.
Besides, the representation $\phi_h$ is then obtained by
\begin{equation} \label{eq:phi}
\phi_h(z_h, a_h)= \int \psi(s_h, a_h) \bar{b}_h^{\mathcal{P}}(z_h)(s_h)\mathrm{d}s_h.
\end{equation}
The representation $\phi_h$ will also be used to calculate the exploration bonus.

% Building upon the centralized representation learning discussed so far, we now shift our focus to decentralized methods required for team settings. In the following section, we define the specific information structure explored in this work.

% The above construction is written for a generic finite-memory information
% state $z_h$. In the centralized reference case, $z_h$ contains the joint
% observation-action history of all members, and the resulting approximate MDP serves as the global reference model. 

Building upon the centralized representation learning discussed so far, next, we shift our focus to decentralized methods required for team settings. 
In the decentralized algorithm, under the delayed
sharing information structure, member $i$ only observes the delayed common
information $c_h$ and its private information $p_h^i$. Therefore, the local finite-memory information state of member $i$ is denoted by
\begin{equation}
z_h^i=(c_h,p_h^i),
\label{eq:local-info-state}
\end{equation}
or, equivalently, by the corresponding finite-memory window contained in
$(c_h,p_h^i)$.

Applying the same finite-memory belief recursion in
\eqref{eq:finite-memory-belief-recursion} to the information available to
member $i$ gives the member-side approximate belief
\begin{equation}
\bar b_h^i(z_h^i).
\label{eq:local-approx-belief}
\end{equation}
This belief is still computed using the shared dynamics representation $(\omega,\psi)$, but the conditioning information is now local to member $i$.
Thus, the difference between the centralized construction and the decentralized construction lies in the information used to form the belief, not in the latent dynamics representation.
The member-side feature used by Algorithm~\ref{alg:ma-porl-jr} is then defined as
\begin{equation}
\phi_h^i(z_h^i,a_h^i)
=
\int_S \psi^i(s_h,z_h^i,a_h^i)
\bar b_h^i(z_h^i)(s_h)\,ds_h,
\label{eq:local-feature}
\end{equation}
and
\begin{equation}
\mu^i(o_{h+1}^i)
=
\int_S \omega(s')^\top O^i(o_{h+1}^i\mid s')\,ds' .
\label{eq:local-mu}
\end{equation}
Here $\psi^i$ denotes the effective local state-action feature induced by the
shared feature $\psi$ under member $i$'s available information. 
%%%%%%%
% When the original low-rank feature \(\psi\) is defined on the
% joint action \(a_h=(a_h^i,a_h^{-i})\), \(\psi_h^i\) is obtained by
% marginalizing over the unavailable actions of the other members under the
% current joint policy, namely
% \[
% \psi_h^i(s,z_h^i,a_h^i)
% =
% \mathbb{E}\!\left[
% \psi(s,(a_h^i,A_h^{-i}))
% \mid z_h^i
% \right].
% \]
%%%%%
Equivalently, we write
\begin{equation}
(\phi_h^i,\mu^i)=q^i(\omega,\psi),
\label{eq:local-q-map}
\end{equation}
where $q^i$ is the local version of the mapping $q$ obtained by replacing the
full finite-memory information state $z_h$ with the local information state
$z_h^i=(c_h,p_h^i)$.

At each iteration $k$, for the decentralized implementation, after the shared dynamics representation
$(\omega_k,\psi_k)$ has been learned from the delayed common information as in~\eqref{eq:representation_learning}, each
member $i$ constructs its own local representations $\phi_{k,h}^i,\mu_k^i$, which is then used in the value iteration.

\section{Convergence analysis}
We first provide, for the reader's convenience, some technical results 
{established in prior works, which will be used in our subsequent analysis.}
% used in our analysis and available in the literature.
\subsection{Technical Lemmas}
The approximate MDP $\mathcal{M}$ defined in \eqref{eq:MDP} retains the structure of low-rank POMDP, and we have the following results.
\begin{lemma} \cite{guo2023provably} \label{lem:L_memory}
For any $\epsilon_1 > 0$, there exists an $L$-structured MDP $\mathcal{M}$ defined in \eqref{eq:MDP} with
\begin{equation}
L = O\!\left(\gamma^{-4}\log\!\bigl(d/\epsilon_1\bigr)\right),
\end{equation}
such that for all $\pi $ and $h \in [H]$,
\begin{align}
&\mathbb{E}_{a_{1:h},\,o_{2:h+1}\sim \pi}\!\left[
\left\| p^{\mathcal{M}} \!\left(o_{h+1}\mid z_h,a_h\right)
- p^{\mathcal{P}} \!\left(o_{h+1}\mid o_{1:h},a_{1:h}\right)
\right\|_{1}
\right] \nonumber
\\& \le \epsilon_1 .
\end{align}

Then, we have that the conditional probability $\mathcal{P}(o' \mid z,a)$ is approximately low-rank. 
Next, we define the value function under $\mathcal{M}$ as
\[ V^{\pi,\mathcal{M},r} = \mathbb{E}^{\pi,\mathcal{M}}
\left[ \sum_{h'=1}^{H} r_{h'} \right].
\]
Hence, for an $L$-memory policy $\pi$, the
value function of $\pi$ in $\mathcal{M}$ can effectively approximate the value function under $\mathcal{P}$, for any policy $\pi$, we have
\begin{equation} \label{eq:V_P_M_deviation}
\left| V_{1}^{\pi,\mathcal{P},r}(o_1)
- V_{1}^{\pi,\mathcal{M},r}(o_1) \right|
\le \frac{H^2 \epsilon_1}{2}.
\end{equation}
Also, for any $\pi, h$ we have
\begin{equation} \label{eq:P_M_deviation}
\left\|  d_h^{\pi,\mathcal P} - d_h^{\pi,\mathcal M} \right\|_{\mathrm{TV}} \le h\epsilon_1.
\end{equation}

\end{lemma}

We also give the {\it Simulation Lemma} commonly used in the related literature.
\begin{lemma} \cite{uehara2021representation} \label{lem:simulation_lemma}
Given two MDPs $(P', r+b)$ and $(P, r)$, for any policy $\pi$, we have
\begin{align}
&V^{\pi}_{P',\,r+b} - V^{\pi}_{P,\,r} 
\nonumber
\\& = \sum_{h=1}^{H} \mathbb{E}_{(s_h,a_h)\sim d^{\pi}_{P'}}
\Bigl[ b_h(s_h,a_h) + \mathbb{E}_{P'(s'_h\mid s_h,a_h)}
\!\left[V^{\pi}_{P,\,r,\,h+1}(s'_h)\right]
\nonumber
\\& - \mathbb{E}_{P(s'_h\mid s_h,a_h)} \left[V^{\pi}_{P,\,r,\,h+1}(s'_h)\right]
\Bigr],
\end{align}
and
\begin{align}
& V^{\pi}_{P',\,r+b} - V^{\pi}_{P,\,r} \nonumber
= \sum_{h=1}^{H}
\mathbb{E}_{(s_h,a_h)\sim d^{\pi}_{P,h}} \nonumber
\\& \Bigl[ b_h(s_h,a_h) +
\mathbb{E}_{P'(s'_h\mid s_h,a_h)} \left[V^{\pi}_{P',\,r+b,\,h+1}(s'_h)\right] \nonumber
\\&  - \mathbb{E}_{P(s'_h\mid s_h,a_h)}
\!\left[V^{\pi}_{P',\,r+b,\,h+1}(s'_h)\right]
\Bigr].
\end{align}
\end{lemma}

We also have the $L$-step back inequality for the true and the learned model as shown in the following lemma.
\begin{lemma}\cite{guo2023provably} \label{lem:L_step_true_model}
Consider a set of functions $\{g_h\}_{h=0}^{H}$ that satisfies
\[ g_h \in \mathcal Z \times \mathcal A \to \mathbb R,
\quad \|g_h\|_\infty \le B,\ \forall\, h\in[H].
\]
Then, for any policy $\pi$, we have
\begin{align}
&\sum_{h=1}^{H}\mathbb E_{\pi}^{\mathcal P}\!\left[g(z_h,a_h)\right]\nonumber
\\&\le
\sum_{h=1}^{H}
\mathbb E_{ z_{h-L-1},\,a_{h-L-1}\sim\pi}^{\mathcal P} \notag\\&
\Biggl[
\left\|
\phi^{\top}(z_{h-L-1},a_{h-L-1})
\right\|_{\beta_{h-L-1}^{-1},\phi_{h-L-1}}
\notag\\&
\sqrt{
|A|^{L}k\cdot
\mathbb E_{(\tilde z_h,\tilde a_h)\sim \gamma_h}^{\mathcal P}
\!\left[g(\tilde z_h,\tilde a_h)^2\right]
+
B^2\lambda_k d
+
kB^2\epsilon_1
}
\Biggr]
+
B\epsilon_1 .
\end{align}
\end{lemma}

\begin{lemma} \cite{guo2023provably} \label{lem:L_step_learned_model}
Consider a set of functions $\{g_h\}_{h=0}^{H}$ that satisfies
\[
g_h \in \mathcal Z \times \mathcal A \to \mathbb R,
\|g_h\|_\infty \le B,\ \forall\, h\in[H].
\]
Then, for any policy $\pi$, we have
\begin{align}
&\sum_{h=1}^{H} \mathbb E_{\pi}^{\hat{\mathcal P}}\!\left[g(z_h,a_h)\right]
\notag
\\&
\le
\sum_{h=1}^{H}
\mathbb E_{\tilde z_{h-L-1},\,a_{h-L-1}\sim\pi}^{\hat{\mathcal P}}
\notag\\&
\Biggl[
\left\|
\tilde\phi^{\top}(z_{h-L-1},a_{h-L-1})
\right\|_{\rho_{h-L-1}^{-1},\,\hat\phi}
\notag
\\&
\sqrt{
|A|^{L}k\cdot
\mathbb E_{(\tilde z_h,\tilde a_h)\sim \rho_h}^{{\mathcal P}}
\!\left[g(\tilde z_h,\tilde a_h)^2\right]
+
B^2\lambda_k d
+
kB^2\epsilon_1
} + B\epsilon_1
\Biggr].
\end{align}
\end{lemma}

The following two lemmas are also used in our analysis.
\begin{lemma}\cite{uehara2021representation}\label{lem:semidefinite_matrix}
Consider the following process. For $n=1,\ldots,N$, let
\[ M_n = M_{n-1} + G_n, \qquad M_0 = \lambda_0 I,
\]
where $G_n$ is a positive semi-definite matrix with eigenvalues upper-bounded by $1$. Then we have
\[ 2 \log \det(M_N) - 2 \log \det(\lambda_0 I) \ge
\sum_{n=1}^N \operatorname{Tr}\!\bigl(G_n M_{n-1}^{-1}\bigr).
\]
\end{lemma}
\begin{lemma}\cite{uehara2021representation}\label{lem:matrix_eigenvalue}
Suppose $\operatorname{Tr}(G_n) \le B^2$, where $G_n$ is a positive semidefinite matrix with eigenvalues upper-bounded by $1$. Then
\[ 2 \log \det(M_N) - 2 \log \det(\lambda_0 I) \le d \log(1 + \frac{N B^2}{d \lambda_0}). \]
\end{lemma}

% We define the following notations that will be used extensively in the analysis.

% For any $k$, we define $\bar{\pi}_k$ to be
% \[
% \bar{\pi}_k = \sum_{i=1}^k \pi^k / k.
% \]
% For a fixed $(k,h)$, $\rho_h^k$ is the distribution on $\mathcal{Z}\times\mathcal{A}$ induced by applying $\{\pi_i^k\}_{i=1}^k$ and then doing uniformly random actions for $L$ times.
% We define the distribution of $(z,a)$ for any $h$, denoted by $\rho_h^k \in \Delta(\mathcal{Z}\times\mathcal{A})$, as follows:
% \[
% \rho_h^k(z,a)=d^{\bar{\pi}_k \circ_L U(\mathcal{A})}_{\rho,h}(z,a).
% \]

% Similarly, we define the distribution on $\mathcal{Z}\times\mathcal{A}$ after doing uniformly random actions for $2L$ times.
% For any $k,h\ge 1$, we define $\beta_h^k$ as follows:
% \[
% \beta_h^k(z,a)=d^{\bar{\pi}_k \circ_{2L} U(\mathcal{A})}_{\rho,h}(z,a).
% \]

% We also define the distribution induced by $\bar{\pi}_k$.
% For any $k,h$, we also define $\gamma_h^k \in \Delta(\mathcal{S}\times\mathcal{A})$ as follows:
% \[
% \gamma_h^k(z,a)=d^{\bar{\pi}_k}_{\rho,h}(z,a).
% \]

% For notational simplification, we denote
% \[
% \|x\|_{\rho,\phi}=\|x\|_{\Sigma_{\rho,\phi}},
% \qquad
% \|x\|_{\rho^{-1},\phi}=\|x\|_{\Sigma_{\rho,\phi}^{-1}}
% \]
% for $x\in\mathbb{R}^d$ and $\phi\in\Phi$, where
% \[
% \Sigma_{\rho,\phi}
% =
% \mathbb{E}_{z,a\sim \rho}\!\left[\phi(z,a)\phi^\top(z,a)\right]+\lambda I.
% \]
% We define $\|\cdot\|_{\beta,\phi}$ and $\|\cdot\|_{\gamma,\phi}$ in the same way.

Next, we introduce several mixture-induced occupancy measures that will be repeatedly used in the subsequent analysis.
For each episode index $k$, define the averaged policy
\begin{align}
\bar{\pi}_k = \frac{1}{k}\sum_{\ell=1}^k \pi^{\ell} .
\end{align}
Based on $\bar{\pi}_k$, we define a family of distributions over $\mathcal Z \times \mathcal A$. In particular, for any pair $(k,h)$, let $\rho_h^k \in \Delta(\mathcal Z \times \mathcal A)$ denote the occupancy measure at stage $h$ induced by following $\bar{\pi}_k$ and subsequently taking uniformly random actions for $L$ steps, namely,
\[
\rho_h^k(z,a) = d_{\rho,h}^{\,\bar{\pi}_k \circ_L U(\mathcal A)}(z,a).
\]
Analogously, we define $\beta_h^k \in \Delta(\mathcal Z \times \mathcal A)$ as the corresponding occupancy measure associated with $2L$ additional uniformly random actions:
\[
\beta_h^k(z,a)
=
d_{\rho,h}^{\,\bar{\pi}_k \circ_{2L} U(\mathcal A)}(z,a).
\]
In addition, we let $\gamma_h^k$ denote the stage-$h$ occupancy measure generated directly by the averaged policy $\bar{\pi}_k$, without any appended random exploration. That is,
\[\gamma_h^k(z,a)= d_{\rho,h}^{\,\bar{\pi}_k}(z,a).
\]

Given any $x\in\mathbb R^d$ and $\phi\in\Phi$, define
\[ \|x\|_{\rho,\phi} = \|x\|_{\Sigma_{\rho,\phi}}, \quad
\|x\|_{\rho^{-1},\phi} = \|x\|_{\Sigma_{\rho,\phi}^{-1}},
\]
where the regularized second-moment matrix under distribution $\rho$ is given by
\[  \Sigma_{\rho,\phi} = \mathbb E_{(z,a)\sim \rho}\left[\phi(z,a)\phi^\top(z,a)\right]+\lambda I.
\]
The quantities $\|\cdot\|_{\beta,\phi}$ and $\|\cdot\|_{\gamma,\phi}$ are defined in the same manner.

%%%%%%%%%%%%%%%%%%%%%%%%%%%%%%%%%%%%%%%
\subsection{Main Results}
In the following, we study the relation between the optimal policies of the members and the team.
{
We first show the relation between the global information state and the local information state under the same approximate model.
}
\begin{lemma} \label{lem:shared_model_monotone_alt}
Fix a step $h$ and a member $i$. 
Let $\Delta_h$ denote the common information, and let $\Lambda_h^i$ and $\Lambda_h^{1:N}$ denote the private information of member $i$ and the collection of all members' private information, respectively. 
Define $\Lambda_h^{-i} = \Lambda_h^{1:N}\setminus \Lambda_h^i$.
All members hold a surrogate probabilistic model $\hat{\mathcal M}$.
Define the local and global predictive information states under the \emph{same} system model $\hat{\mathcal M}$ as
\begin{equation}
\Pi_h^{i}(s') = P^{\hat{\mathcal M}}(S_{h+1}=s'\mid \Delta_h, \Lambda_h^{i}),
\label{eq:Pi_local_def}
\end{equation}
\begin{equation}
\Pi_h^{g}(s') = P^{\hat{\mathcal M}}(S_{h+1}=s'\mid \Delta_h, \Lambda_h^{1:N}).
\label{eq:Pi_global_def}
\end{equation}

Then there exists a nonnegative reweighting function
\begin{equation} \label{eq:weights}
w_h^{i}(s') = P^{\hat{\mathcal M}}(\Lambda_h^{-i}\mid S_{h+1}=s',\Delta_h,\Lambda_h^i) \ge 0,
\end{equation}
such that the global predictive information state is obtained by a normalized positive reweighting of the local one
\begin{equation} \label{eq:reweighting_alt}
\Pi_h^{g}(s') = \frac{w_h^{i}(s')\,\Pi_h^{i}(s')}
{\sum_{\bar s\in\mathcal S} w_h^{i}(\bar s)\,\Pi_h^{i}(\bar s)}.
\end{equation}
% In particular, the global predictive information state is absolutely continuous with respect to the local one: $\Pi_h^{g}(s')=0$ whenever $\Pi_h^{i}(s')=0$. Hence $\Pi_h^{g}$ is uniquely determined by the local information state $\Pi_h^{i}$ together with the missing-information likelihood $w_h^{i}$.
% Moreover, for each $s'\in\mathcal S$, $\Pi_h^{g}(s')$ is strictly increasing in the corresponding local mass $\Pi_h^{i}(s')$ (holding $w_h^{i}$ fixed), with
% \begin{equation} \label{eq:mono}
% \frac{\partial \Pi_h^{g}(s')}{\partial \Pi_h^{i}(s')}
% = \frac{w_h^{i}(s')} {\sum_{\bar s\in\mathcal S} w_h^{i}(\bar s),  \Pi_h^{i}(\bar s)} > 0.
% \end{equation}
\end{lemma}

\begin{proof}
Under the fixed surrogate model $\hat{\mathcal M}$, the probabilities $P^{\hat{\mathcal M}}(\Lambda_h^{-i}\mid \Delta_h,\Lambda_h^i)>0$ and
$P^{\hat{\mathcal M}}(\Lambda_h^{-i}\mid S_{h+1}=s',\Delta_h, \Lambda_h^i)>0$ for all $s'\in\mathcal S$.
 In particular, the weights in~\eqref{eq:weights} are nonnegative and the denominator in~\eqref{eq:reweighting_alt} is strictly positive.

Conditioning on all members' information $(\Delta_h,\Lambda_h^{i},\Lambda_h^{-i})$ and applying Bayes' rule to $\Lambda_h^{-i}$ given the local conditioning $(\Delta_h,\Lambda_h^{i})$,
\begin{equation}
\begin{aligned}
&P^{\hat{\mathcal M}}(S_{h+1}=s'\mid \Delta_h,\Lambda_h^{i},\Lambda_h^{-i}) \\
&= \frac{ P^{\hat{\mathcal M}}(\Lambda_h^{-i}\mid S_{h+1}=s',\Delta_h,\Lambda_h^{i})\,
P^{\hat{\mathcal M}}(S_{h+1}=s'\mid \Delta_h,\Lambda_h^{i})
}{ P^{\hat{\mathcal M}}(\Lambda_h^{-i}\mid \Delta_h,\Lambda_h^{i}) },
\end{aligned}
\label{eq:bayes_step}
\end{equation}
where, by the definitions of $w_h^{i}$ and $\Pi_h^{i}$, the numerator equals $w_h^{i}(s')\,\Pi_h^{i}(s')$.
The normalizing constant in the denominator follows from the law of total probability over $S_{h+1}$,
\begin{equation}
\begin{aligned}
&P^{\hat{\mathcal M}}(\Lambda_h^{-i}\mid \Delta_h,\Lambda_h^{i})
\\&= \sum_{\bar s\in\mathcal S}  P^{\hat{\mathcal M}}(\Lambda_h^{-i}\mid S_{h+1}=\bar s,\Delta_h,\Lambda_h^{i}) P^{\hat{\mathcal M}}(S_{h+1}=\bar s\mid \Delta_h,\Lambda_h^{i})
\\&= \sum_{\bar s\in\mathcal S} w_h^{i}(\bar s)\,\Pi_h^{i}(\bar s).
\end{aligned}
\label{eq:norm_step}
\end{equation}
Substituting the numerator of~\eqref{eq:bayes_step} and the denominator~\eqref{eq:norm_step} yields~\eqref{eq:reweighting_alt}. 
%
% By Bayes' rule under the same approximated
% model $\hat{\mathcal M}$,
% \begin{equation}
% \begin{aligned}
% &P^{\hat{\mathcal M}}(S_{h+1}=s'\mid \Delta_k,\Lambda_h^{i},\Lambda_k^{-i}) \\& = \frac{ P^{\hat{\mathcal M}}(\Lambda_k^{-i}\mid S_{h+1}=s',\Delta_k,\Lambda_h^{i})
% P^{\hat{\mathcal M}}(S_{h+1}=s'\mid \Delta_k,\Lambda_h^{i})
% }{ P^{\hat{\mathcal M}}(\Lambda_k^{-i}\mid \Delta_k,\Lambda_h^{i}) }.
% \end{aligned}
% \end{equation}
% Noting that
% \begin{equation}
% \begin{aligned}
% &\ P^{\hat{\mathcal M}}(\Lambda_k^{-i}\mid \Delta_k,\Lambda_h^{i})
% \\&= \sum_{\bar s\in\mathcal S}  P^{\hat{\mathcal M} }(\Lambda_k^{-i}\mid S_{h+1}=\bar s,\Delta_k,\Lambda_h^{i})\;
%  P^{\hat{\mathcal M}}(S_{h+1}=\bar s\mid \Delta_k,\Lambda_h^{i}),
% \end{aligned}
% \end{equation}
% we obtain~\eqref{eq:reweighting_alt}. The stated derivative follows by direct differentiation with $w_h^{i}$ fixed.
\end{proof}
{
The following theorem should be interpreted conditionally on the learned surrogate model. The representation learning phase first estimates the surrogate model from delayed common information. Once this model is fixed, the subsequent planning phase becomes a team decision problem with a known probabilistic model, to which the structural
results of \cite{malikopoulos2022team} apply after verifying that the induced information states satisfy the hypotheses of \cite[Theorem~7]{malikopoulos2022team}.}

\begin{theorem} \label{theo:ctde_consistency}
At iteration $k$, denote $V^{\pi, \hat{\mathcal{M}}_k, r+\tilde{b}_k}$ as the value function obtained by the global information using value iteration under the approximate system model $\hat{\mathcal{M}}_k$, $ \pi^{i}_k $ is obtained by Algorithm~\ref{alg:ma-lsvi-jr} with the local information state, then
\begin{equation} \label{eq:global_local_consist}
\{ \pi^{i}_k \}_{i=1}^{N}  \in \arg\max_{\pi}\, V^{\pi, \hat{\mathcal{M}}_k,\,r+\tilde{b}_k}
\end{equation}
% Then each member's local dynamic programming on its predictive information state $\Pi_t^i$ corresponds to the same centralized objective, 
The optimal local decision rule is a best response with respect to the centralized team return conditioned on the global information state $\Pi_k$.
\end{theorem}

\begin{proof}
At iteration $k$, Algorithm~1 learns the low-rank dynamics
representation $(\hat{\omega}_k,\hat{\psi}_k)$ from the delayed
common information by solving the maximum-likelihood problem in~\eqref{eq:representation_learning}. Consequently, all members share the same learned latent
dynamics representation and therefore the same surrogate
probabilistic model, which we denote by $\hat{M}_k$. During the
subsequent policy computation, this surrogate model is fixed, and
thus the dynamic programming recursion is carried out with respect
to a known probabilistic model.

Each member then applies the local mapping
\[
(\hat{\phi}_{k,h}^{\,i},\hat{\mu}_k^{\,i})
=
q^{i}(\hat{\omega}_k,\hat{\psi}_k),
\]
which combines the shared surrogate dynamics with the information
available to member $i$, namely the delayed common information
$c_h$ and its private information $p_h^{i}$. Hence, the resulting
member-side information state is
\[
\bar b_h^{\,i}(z_h^{\,i}),
\qquad
z_h^{\,i}=(c_h,p_h^{\,i}),
\]
and is computed under the common surrogate model $\hat M_k$.

Since the surrogate dynamics are fixed during the planning stage,
the member-side information state depends only on the available
information $(c_h,p_h^{\,i})$ and not on member $i$'s own control
strategy. Therefore, the information state satisfies the structural
property established in \cite[Theorem~5]{malikopoulos2022team}. The
difference among members lies only in the conditioning information,
not in the underlying probabilistic model.

Furthermore, Lemma~7 establishes that, under the same surrogate
model, the manager's predictive information state is obtained from
the member's predictive information state through a normalized
positive reweighting,
\[
\Pi_h^{g}(s')
=
\frac{w_h^{i}(s')\,\Pi_h^{i}(s')}
{\sum_{\bar s}w_h^{i}(\bar s)\Pi_h^{i}(\bar s)}.
\]
This is precisely the relationship required in
\cite[Lemma~7]{malikopoulos2022team} to recover the manager's
information state from each member's information state.

Consequently, once the surrogate model has been learned, all
assumptions required by \cite[Theorem~7]{malikopoulos2022team} are
satisfied. In particular,

\begin{enumerate}
\item all members plan with respect to the same probabilistic model
      $\hat M_k$;

\item each member possesses an information state that is
      independent of its own control strategy and evolves according
      to a strategy-independent recursion;

\item the manager's information state is recoverable from the
      corresponding member's information state.
\end{enumerate}

Therefore, the manager--member equivalence established in
\cite[Theorem~7]{malikopoulos2022team} applies directly to the surrogate
team problem. The member-side Bellman recursion solved by
Algorithm~2 therefore yields exactly the corresponding component of
the manager's optimal separated control law.

Specifically, member $i$ evaluates
\[
\hat P^{\,i}(o_{h+1}^{\,i}\mid z_h^{\,i},a_h^{\,i})
=
\hat\mu_k^{\,i}(o_{h+1}^{\,i})^\top
\hat\phi_{k,h}^{\,i}(z_h^{\,i},a_h^{\,i}),
\]
which is induced by the common surrogate model $\hat M_k$ and the
member's local information state. The Bellman recursion therefore
computes the conditional expected team return under the same
surrogate model used by the manager, differing only in the
conditioning information.

Hence the collection of member-side policies satisfies
\[
\{\pi_k^{\,i}\}_{i=1}^{N}
\in
\arg\max_{\pi}
V^{\pi,\hat M_k,r+\tilde b_k},
\]
which establishes the desired global--local equivalence.
\end{proof}

{
\begin{remark}\label{remark:inconsistency}
The consistency result relies on the exact Bellman backup, where the expectation is taken with respect to the true observation distribution over all possible observations. In our implementation of the proposed Algorithm~\ref{alg:ma-porl-jr}, this expectation is approximated by an
empirical average over Monte Carlo samples stored in the replay buffer. As a result, the consistency is approximate rather than exact, and the discrepancy depends on the finite-sample estimation error and the coverage of the collected data.
\end{remark}
}

Because of the global-local equivalence shown in Theorem~\ref{theo:ctde_consistency}, next, we focus on the optimality error of the solution obtained by the global information state.
The following theorem shows the sample complexity of the proposed decentralized algorithm. 
\begin{theorem} \label{thm:distributed_team_pac}
Let $\delta, \epsilon \in (0,1)$ be given.  
For each episode $k\in\{1,\dots,K\}$, let member $i$'s local policy obtained by Algorithm~\ref{alg:ma-lsvi-jr} be
\begin{equation}
\pi^i_k \in \Pi^i.
\end{equation}
The induced team policy is defined by
\[
\pi_k = (\pi^1_k,\dots,\pi^N_k).
% \quad
% \Pi^{\mathrm{team}} = \Pi^1 \times \cdots \times \Pi^N.
\]
Let $ \bar{\pi} = \frac{1}{K}\sum_{\ell=1}^K \pi_{\ell}$  be the uniform mixture of $\pi_1,\dots,\pi_{K}$, and let
\[ \pi^\star \in \arg\max_{\pi } V^{\mathcal{P},\pi,r} \]
be the optimal centralized team policy.

Choose the parameters as
\begin{align} \label{eq:parameters}
&\alpha_k = \tilde{\Theta}\!\left(\sqrt{k|\mathcal A|^L \zeta_k + \lambda_k d  }\right),\nonumber
\\&
\lambda_k = \Theta\!\left(d\log(|\mathcal F|k(H-n)/\delta)\right),
\nonumber
\\&
\epsilon_1 = \Theta\!\left(
\frac{\epsilon}{
H^2 d^{1/2}\gamma^{-4}\log(1/\epsilon)\log(dH|\mathcal F|/\delta)^{1/2}}
\right), \nonumber
\\&
L = \Theta\!\left(\gamma^{-4}\log(d/\epsilon_1)\right), \nonumber
\\& 
\zeta_k = \Theta\!\left(\frac{\log(|\mathcal F|k(H-n)/\delta)}{k(H-n)}\right).
\end{align}

% [MODIFIED]: Added the irreducible delay penalty to the final sub-optimality bound.
Then, with probability at least $1-\delta$, we have
\begin{equation}
V^{\mathcal{P},\pi^\star,r} - V^{\mathcal{P},\bar{\pi},r} \le \epsilon + O\Big( n\sqrt{ \frac{|A|^L}{d(H-n)  }+d }\Big),
\end{equation}
after
\begin{align}
&HK \nonumber\\& = O\Big(
\frac{H^5}{\epsilon^2}
\Big( \frac{|A|^{2L}d^2}{H-n} + |A|^L d^4 \Big)
\log^2\Big(
\frac{d(H-n)|\mathcal F|}{\delta} \frac{H^4}{\epsilon^2}
\nonumber\\&
\quad (\frac{|A|^{2L}d^2}{H-n} + |A|^L d^4)
\Big)
\nonumber\\
& \quad +
\frac{HL^2}{\epsilon^2}
\Big( \frac{|A|^L}{H-n}+d^2 \Big)
\log\!(
\frac{d(H-n)|\mathcal F|}{\delta}
\cdot
\frac{L^2}{\epsilon^2}
\nonumber\\
&\quad ( \frac{|A|^L}{H-n}+d^2 ) )
\Big).
\end{align}
episodes of interaction with the environment.
\end{theorem}
\begin{proof}
    The detailed proof is in the Appendix.
\end{proof}

\begin{remark}
{
Theorem~\ref{thm:distributed_team_pac} is important and provides a {\it finite-sample performance guarantee} for the decentralized policies generated by Algorithm~\ref{alg:ma-porl-jr}, showing that the locally computed member-side policies achieve near-optimal team performance under the original partially observable model. The bound reflects three sources of error. The finite-memory truncation replaces the full history-dependent belief by an $L$-memory information state, introducing the approximation error $\epsilon_1$, which decreases with $L$ at the cost of increased storage complexity (see Lemma~\ref{lem:L_memory}). The low-rank representation learning contributes the statistical error $\zeta_k$ (see Lemma~\ref{lem:MLE}), which decreases with the number of samples. The third is the error caused by delayed information sharing: since joint information is available only up to $H-n$, the last $n$ stages use a frozen bonus, which yields the delay-dependent term in the final bound. Thus $\epsilon_1$ and $\zeta_k$ govern, respectively, the bias and the statistical accuracy of the learned model, while the delay governs the loss from outdated common information.

The proof combines three ingredients: the finite-memory analysis follows~\cite{guo2023provably}, the statistical analysis of $(\widehat\omega_k,\widehat\psi_k)$ follows~\cite{uehara2021representation}, and the reduction from the $N$ member-side problems to a single centralized analysis follows from the team-theoretic correspondence (Theorem~\ref{theo:ctde_consistency}). The new component is the treatment of delayed common information: \eqref{eq:bouns_error} bounds the deviation $\epsilon_{\mathrm{delay}}$ of the frozen bonus, which the value-difference decomposition propagates into the delay-dependent term in~\eqref{eq:V_error}.
{
We also note that this bound may be conservative, as it is derived from worst-case upper bounds. It should therefore be interpreted primarily as a theoretical characterization of the scaling with the problem parameters, rather than as a tight prediction of the number of episodes required in practice.
}
}
\end{remark}

\section{Numerical Experiments}
We evaluate our algorithm on a multi-agent combination lock environment, a cooperative Dec-POMDP that extends the single-agent combination lock in~\cite{guo2023provably} to $N=3$ agents. The environment is designed to test an algorithm's ability to perform joint exploration and representation learning under partial observability.

The environment maintains a latent state $s \in \{0, 1, 2\}$, where states $0$ and $1$ are "good" states and state $2$ is an absorbing "bad" state, the horizon $H = 5$.
At the beginning of each episode, the latent state is initialized as $s_0 \sim \text{Bernoulli}(p_{\text{switch}})$, taking value $0$ or $1$ with equal probability.  
% Once the environment transitions into state $2$, it remains there for the remainder of the episode.

At each step $h \in \{0, \ldots, H-1\}$, all $N$ agents simultaneously execute local actions $\mathbf{a}_h = (a_h^1, \ldots, a_h^N)$, where each agent has three actions to choose from. 
% $a_h^i \in \mathcal{A} = {0, \ldots, |\mathcal{A}|-1}$. 
The transition is determined by whether the joint action is collectively correct: for each latent state $s \in {0,1}$, there exists a sequence of optimal joint actions $\mathrm{col}\big( \mathbf{a}_0, \dots, \mathbf{a}_{H-1} \big)$ such that $a_h^{i} = \texttt{opt}s[i][h]$ for agent $i$. If and only if all $N$ agents simultaneously select their correct action, the environment remains in a good state (transitioning between states $0$ and $1$ with probability $p_{\text{switch}}$); otherwise, the environment transitions to the bad state $s=2$.  Each agent in state $2$ independently recovers to a uniform random state in ${0,1}$ with probability $p_{\text{recover}}$ per step.

The reward is shared among all agents. 
The team receives a reward of $1$ if all agents are alive (in state ${0,1}$) after the transition, plus a partial credit of $\frac{0.5n_{\text{correct}}}{N}$ where $n_{\text{correct}}$ is the number of agents that selected the correct action at the current step. The correct action for each agent depends on the neighbor's current state, which switches with probability $p_{\text{switch}}$ at each step.

Each agent receives a local observation of the local state.
% rather than directly observing the global state. 
The observation is constructed by combining a representation of the current latent state with a time-dependent signal and additive noise. This combined signal is then transformed  through an agent-specific linear mapping, so that different agents observe distinct but 
information-preserving projections of the same underlying state.
%
% Furthermore, at even-numbered steps ($h \bmod 2 = 0$), observations
% corresponding to states $0$ and $2$ are masked to zero, preventing agents from distinguishing the good state from the absorbing bad state at those steps.
%
Since the latent state is not directly observable, agents rely on an observation history. 
% Following prior work, we represent the history at
% step $h$ as a concatenation of the previous and current joint observations:
% $$z_h = [o_{h-1}^1 | \cdots | o_{h-1}^N | o_h^1 | \cdots | o_h^N] \in \mathbb{R}^{2 N d_{\text{obs}}}$$
%   This two-step window is shared across all agents during centralized trai ning.

%%%%%%%%%%%%%%%%%%%%%%%%%%%%%%%%%%%%%%%%%%%%%%%%%%%%
% The environment presents three compounding challenges: 
% \begin{itemize}
% \item The optimal action at each step depends on the latent state, which is hidden and must be inferred from a noisy, rotated observation;
% \item Partial observability is heterogeneous, each agent's view is a distinct linear transformation of the same latent signal, so no single agent can recover the full state alone;
% \item The cooperative constraint requires all agents to act correctly simultaneously, meaning a single agent's mistake collapses the entire episode, making joint exploration substantially harder than in the single-agent case.
% \end{itemize}

We first compare the performance of the proposed decentralized method under different communication delays. The results are shown in Fig.~\ref{fig:delay_return}. As the delay increases, the average return decreases, since each agent makes decisions based on less recent common information. This illustrates the impact of delayed information on decentralized coordination.
We also evaluate the consistency between the decentralized solution and the centralized solution. Specifically, the centralized solution is computed using the full joint information, while the decentralized solution is obtained using $2$-step delayed common information and each agent's private local observation. 

We also evaluate the consistency of the members' local optimal policy and the manager's global optimal policy.
 Fig.~\ref{fig:centralized_decentralized_consistency} reports the agreement rate at each decision step, averaged across the three agents and over the last 10 training iterations. 
 It can be seen that the decentralized solution agrees with the centralized solution in most decision steps, indicating that the proposed decentralized representation can approximate the centralized team decision reasonably well. 
As discussed in Remark~\ref{remark:inconsistency}, the agreement is not exact because, rather than summing over all possible next observations as in the theoretical formulation, our implementation approximates this sum via a Monte Carlo average over sampled trajectories. The residual disagreement observed in Fig.~\ref{fig:centralized_decentralized_consistency} is therefore consistent with this empirical approximation, and is expected to vanish as the sample coverage of the observation space improves.

% {\color{red}
% Add a remark or discussion on the inconsistency, and recall that in the simulation. 
% Make a formula analysis, and ask for some feedback from other reviewers.
% }

\begin{figure}
    \centering
    \includegraphics[scale=0.4]{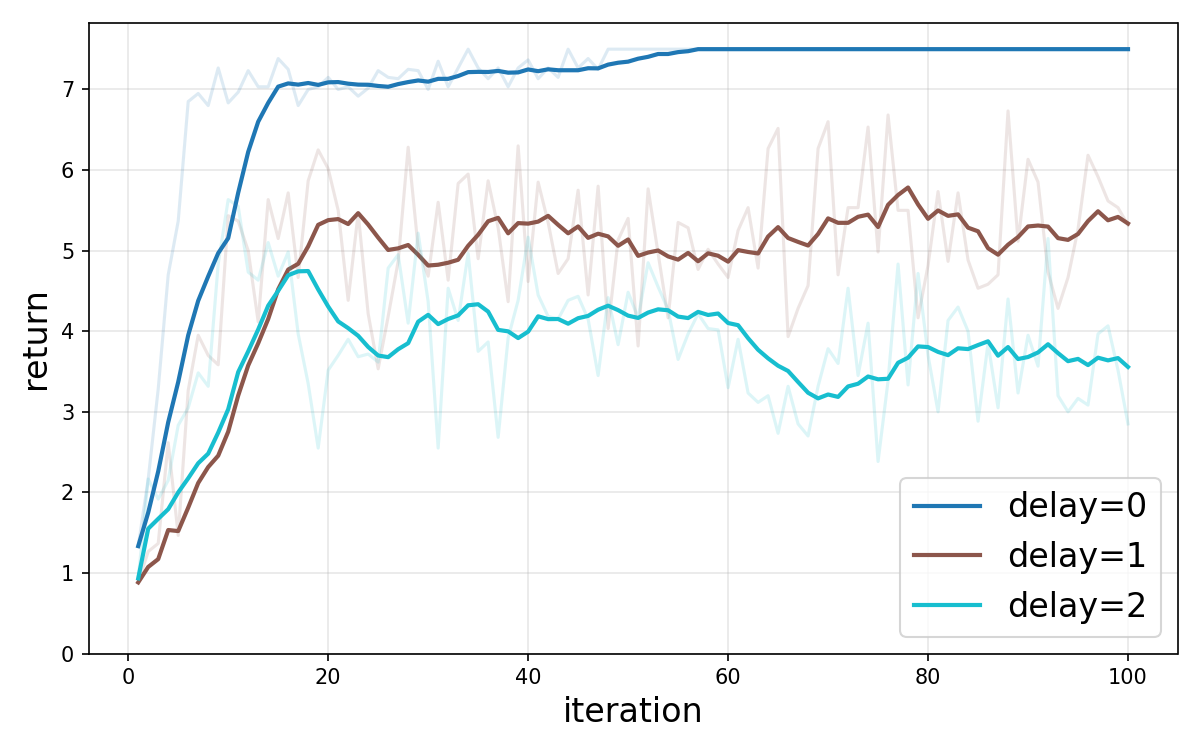}
    \caption{Return under different delays}
    \label{fig:delay_return}
\end{figure}
\begin{figure}
    \centering
    \includegraphics[scale=0.5]{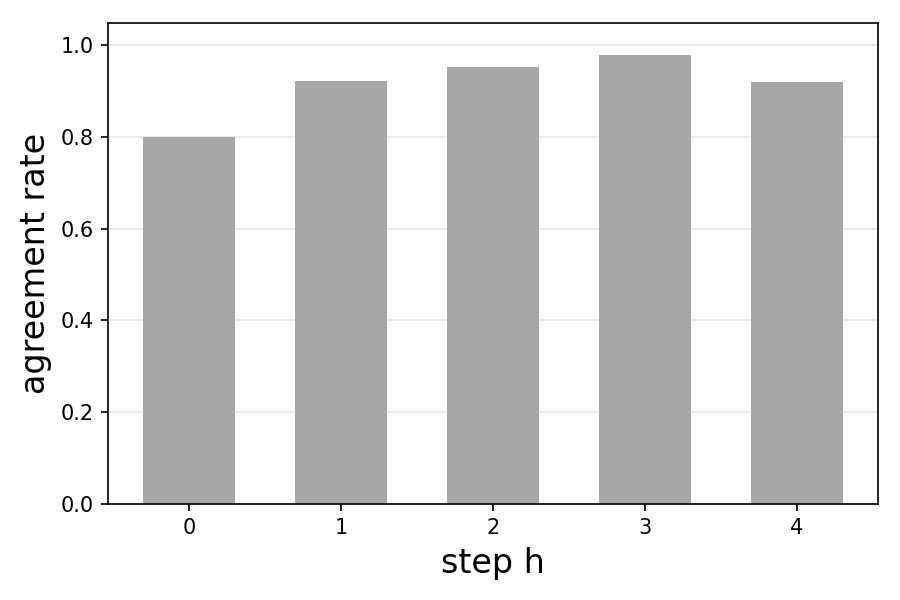}
    \caption{Centralized-decentralized consistency}
    \label{fig:centralized_decentralized_consistency}
\end{figure}

\section{Concluding Remarks}

In this paper, we studied decentralized partially observable team decision problems with delayed information sharing and unknown dynamics. We proposed a fully decentralized learning and planning framework that combines team-theoretic equivalence with low-rank representation learning. In the proposed method, each team member uses delayed common information and local private information to construct an approximate low-rank Markov decision process and compute its policy through value iteration, without relying on a centralized coordinator.
We showed that the member-side solution corresponds to the associated component of an approximate team-optimal solution under the learned model. We also established a finite-sample performance guarantee that captures the effects of representation learning and delayed information sharing. 
% The multi-agent combination lock example illustrates the potential of the proposed method for decentralized learning in cooperative partially observable settings.
%
Future work includes relaxing the structural assumptions and extending the framework to more general information-sharing patterns and function approximation architectures.

\section{Appendix} \label{appen:proof}

\subsection{Proof of Theorem~\ref{thm:distributed_team_pac}}
First, we provide the following maximum likelihood estimation (MLE) guarantee, which upper-bounds the error using the learned features at any iteration.
Let $\hat P$ denote the learned POMDP model by solving~\eqref{eq:pomdp} and $\hat{\mathcal{M}}$ denote the learned MDP model by solving~\eqref{eq:MDP}.
Since the transition matrix ${P}$ is invariant with time $h$, according to \cite[Lemma 18]{uehara2021representation}, we have the following result. 

\begin{lemma} \label{lem:MLE}
For any $h\in[H]$, let $\rho_h$ denote the joint distribution of
$(o_{3-2L:h},a_{3-2L:h},o_{h+1})$ induced by the dataset $\mathcal D$ of size $k$. Then, with probability at least $1-\delta$, we have
\begin{align}
&\mathbb E_{o_{3-2L:h},\,a_{3-2L:h}\sim \rho_h}
\![ \| \hat P \left(\cdot \mid o_{3-2L:h},a_{3-2L:h}\right)^\top \nonumber
\\& - P\left(\cdot \mid o_{3-2L:h},a_{3-2L:h}\right)^\top \|_1^2 ] \nonumber
\\& \le \zeta_k = O\!\left(\frac{\log(k(H-n)|\mathcal F|/\delta)}{k(H-n)}\right).
\end{align}

In addition, we have
\begin{align}
&\mathbb E_{z_h,a_h\sim \rho_h}
\![\| P^{ \hat{\mathcal M}}\!\left(\cdot \mid z_h,a_h\right)^\top
- P^{\mathcal M}\!\left(\cdot \mid z_h,a_h\right)^\top
\|_1^2 ] \nonumber
\\& \le O\!\left(\frac{\log(k(H-n)|\mathcal F|/\delta)}{k(H-n)}+\epsilon_1\right).
\end{align}
\end{lemma}

Then, we provide the bonus error because of the delayed information structure.
\begin{lemma}
\label{lem:frozen_bonus_error}
For a fixed episode $k$, each stage $h\in[H]$, denote the true bonus by
\[ \hat{b}_{k,h}(z,a)= \min\left\{ \alpha_k \sqrt{ \bigl(\phi_h(z,a)\bigr)^\top \Sigma_h^{-1}\phi_h(z,a) }, 1 \right\},
\]
where $ \Sigma_h = \sum_{(z',a')\sim \mathcal B_h \ } \phi_{k,h}(z',a')\bigl(\phi_{k,h}(z',a')\bigr)^\top +\lambda_k I $, $\mathcal B_h = \mathcal{D}_h \cup \mathcal{D}'_h$.
% \[ \Sigma_h = \sum_{(z',a')\sim \mathcal D_h \ } \phi_{k,h}(z',a')\bigl(\phi_{k,h}(z',a')\bigr)^\top +\lambda_k I. \]
Fix $\bar h= H-n$. For the last $n$ stages $h=\bar h+1,\dots,H$, define the frozen bonus
\[ \tilde b_{k,h}(z,a)=\hat{b}_{k, \bar h}(z,a). \]
Assume that, for all $h,z,a,s$,
\begin{align}
\|\phi_{k,h}(z,a)\|\le L_\phi,
\quad
\|\psi_k(s,a)\|\le L_\psi.
\end{align}
Assume further that there exists a constant $C_b>0$ such that, for all
$h=\bar h+1,\dots,H$,
\begin{align}
\|\bar b_h^{\mathcal{P}}(z)-\bar b_{\bar h}^{\mathcal{P}}(z)\|_1\le C_b.
\end{align}
Then, for every $h\in\{\bar h+1,\dots,H\}$,
\begin{align} \label{eq:bouns_error}
&\bigl|\tilde b_{k,h}(z,a)-\hat{b}_{k,h}(z,a)\bigr|
\nonumber 
\\&\le
\alpha_k
\sqrt{
\frac{2L_\phi L_\psi C_b+2L_\phi^2}{\lambda_k}
} \nonumber \\& 
= \epsilon^k_{\mathrm{delay}} .
\end{align}
\end{lemma}
\begin{proof}
Define
\[ q_h(z,a) =  \bigl(\phi_{k,h}(z,a)\bigr)^\top \Sigma_h^{-1}\phi_{k,h}(z,a).
\]
Then
\begin{align}
&\hat b_{k,h}(z,a)=\min\{\alpha_k\sqrt{q_h(z,a)},\,1\},
\nonumber \\&
\tilde b_{k,h}(z,a)=\min\{\alpha_k\sqrt{q_{\bar h}(z,a)},\,1\}.
\end{align}
Hence
\begin{align}
&\bigl|\tilde b_{k,h}(z,a)- \hat b_{k,h}(z,a)\bigr|
\nonumber \\&
\le \alpha_k \bigl|
\sqrt{q_{\bar h}(z,a)}-\sqrt{q_h(z,a)}
\bigr| 
\nonumber \\& 
\le \alpha_k \sqrt{|q_h(z,a)-q_{\bar h}(z,a)|},
\end{align}
where we used $|\sqrt{x}-\sqrt{y}|\le \sqrt{|x-y|}$ for all $x,y\ge 0$.

For simplicity, write
\[ \phi_h=\phi_{k,h}(z,a), \phi_{\bar h}=\phi_{k,\bar h}(z,a), A_h=\Sigma_h^{-1}, A_{\bar h}=\Sigma_{\bar h}^{-1}.
\]
Then \[ q_h-q_{\bar h} = \phi_h^\top A_h\phi_h-\phi_{\bar h}^\top A_{\bar h}\phi_{\bar h}.
\]
Adding and subtracting $\phi_{\bar h}^\top A_h\phi_{\bar h}$ yields
\begin{align}
q_h-q_{\bar h} = \underbrace{ \phi_h^\top A_h\phi_h-\phi_{\bar h}^\top A_h\phi_{\bar h} }_{T_1} + \underbrace{ \phi_{\bar h}^\top(A_h-A_{\bar h})\phi_{\bar h} }_{T_2}.
\end{align}
Thus,
\[
|q_h-q_{\bar h}|\le |T_1|+|T_2|.
\]

For the first term, we have
\[ |T_1| \le \|A_h\|\,\|\phi_h-\phi_{\bar h}\|(\|\phi_h\|+\|\phi_{\bar h}\|).
\]
Since $\Sigma_h\succeq \lambda_k I$, we have $\forall h$,
\[ \|A_h\|=\|\Sigma_h^{-1}\|\le \frac{1}{\lambda_k}, \]
and therefore
\[ |T_1| \le \frac{2L_\phi}{\lambda_k}\|\phi_h-\phi_{\bar h}\|.
\]
Moreover, by the definition of $\phi_h(z,a)$ and the bound
$\|\psi(s,a)\|\le L_\psi$,
\[
\begin{aligned}
\|\phi_h-\phi_{\bar h}\|
&=
\left\|
\int \psi(s,a)\bigl(\bar b_h^\mathcal{P}(z)(s)-\bar b_{\bar h}^\mathcal{P}(z)(s)\bigr)\,ds
\right\| \\
&\le
\int \|\psi(s,a)\|
\bigl|
\bar b_h^\mathcal{P}(z)(s)-\bar b_{\bar h}^\mathcal{P}(z)(s)
\bigr|\,ds \\
&\le L_\psi \|\bar b_h^\mathcal{P}(z)-\bar b_{\bar h}^\mathcal{P}(z)\|_1 \\
&\le L_\psi C_b.
\end{aligned}
\]
Hence,
\begin{align}
|T_1| \le \frac{2L_\phi L_\psi C_b}{\lambda_k}.
\end{align}
For the second term,
\begin{align} 
|T_2| &\le \|\phi_{\bar h}\|^2\,\|A_h-A_{\bar h}\|
\nonumber \\& 
\le L_\phi^2(\|A_h\|+\|A_{\bar h}\|) \leq \frac{2L_\phi^2}{\lambda_k}.
\end{align}

Combining the bounds for $T_1$ and $T_2$ gives
\begin{align}
|q_h-q_{\bar h}|
\le \frac{2L_\phi L_\psi C_b}{\lambda_k}
+\frac{2L_\phi^2}{\lambda_k}
% +\frac{L_\phi^2 C_D}{\lambda_k^2}.
\end{align}
Substituting this into the previous inequality yields
\begin{align}
&\bigl|\tilde b_{k,h}(z,a) - \hat{b}_{k,h}(z,a)\bigr|
\nonumber \\& \le \alpha_k
\sqrt{ \frac{2L_\phi L_\psi C_b}{\lambda_k} + \frac{2L_\phi^2}{\lambda_k} }.
\end{align}
This completes the proof.
\end{proof}

Next, we move to the proof of the main result.
\paragraph*{Proof of Theorem~\ref{thm:distributed_team_pac}}
According to Theorem~\ref{theo:ctde_consistency}, the decentralized policies correspond to the policy obtained by the centralized method. Therefore, we concentrate the centralized solution in the following analysis.

% [NEW]: Define the ideal bonus vs the actual delayed bonus
For a fixed $k$, let $\widehat{b}_h(z,a)$ denote the ideal exploration bonus computed with the true current features, while $\tilde{b}_h(z,a)$ is the actual bonus used in our algorithm due to the $n$-step delayed information. According to Lemma~\ref{lem:frozen_bonus_error}, the maximum deviation caused by the delayed information structure is bounded:
\begin{equation} \label{eq:delay_bonus_bound}
\max_{h, z, a} |\tilde{b}_h(z,a) - \widehat{b}_h(z,a)| \le \epsilon_{\mathrm{delay}},
\end{equation}
where $\epsilon_{\mathrm{delay}}$ is defined in \eqref{eq:bouns_error}.
Note that for the first $H-n$ steps, $\tilde{b}_h = \widehat{b}_h$. The deviation only accumulates over the last $n$ steps.
    
We first study the value function error under the approximate MDP. By Lemma~\ref{lem:simulation_lemma}, for a fixed $k$,
\begin{align} \label{eq:m_error}
 % [MODIFIED]: Replaced \widehat{b} with \tilde{b} and extracted the delay penalty
 & V^{\pi^*, \widehat{\mathcal{M}}, r+\tilde{b}} - V^{\pi^*, \mathcal{M}, r} \nonumber  
 \\& = \sum_{h=1}^{H} \mathbb{E}_{(z_h, a_h) \sim d_{\widehat{\mathcal{M}}}^{\pi^*}} [ \tilde{b}_h(z_h, a_h)
 \nonumber
 \\&+ \mathbb{E}_{o' \sim {\widehat{\mathcal{M}}}(\cdot|z_h, a_h)} [V_{h+1}^{\pi^*, \mathcal{M}, r}(z_{h+1}')]  \nonumber
 \\&- \mathbb{E}_{o' \sim {\mathcal{M}}(\cdot|z_h, a_h)} [V_{h+1}^{\pi^*, \mathcal{M}, r}(z_{h+1}')] ] \nonumber \\
 & = \sum_{h=1}^{H} \mathbb{E}_{(z_h, a_h) \sim d_{\widehat{\mathcal{M}}}^{\pi^*}} \bigl[ \widehat{b}_h(z_h, a_h) + (\tilde{b}_h(z_h, a_h) - \widehat{b}_h(z_h, a_h)) \nonumber
 \\&+ \mathbb{E}_{o' \sim {\widehat{\mathcal{M}}}(\cdot|z_h, a_h)} [V_{h+1}^{\pi^*, \mathcal{M}, r}(z_{h+1}')] \nonumber \\
 &  - \mathbb{E}_{o' \sim {\mathcal{M}}(\cdot|z_h, a_h)} [V_{h+1}^{\pi^*, \mathcal{M}, r}(z_{h+1}')] \bigr] \nonumber \\
 & \ge \sum_{h=1}^{H} \mathbb{E}_{(z_h, a_h) \sim d_{\widehat{ \mathcal{P}} }^{\pi^*}} \bigl[ \min(c \alpha_k \| \widehat{\phi}_h(z, a) \|_{\Sigma_{\rho_h, \widehat{\phi}_h }^{-1}}, 1) \nonumber
 \\& + \mathbb{E}_{o' \sim {\widehat{\mathcal{P}}}(\cdot|z_h, a_h)} [V_{h+1}^{\pi^*, \mathcal{M}, r}(z_{h+1}')]  \nonumber 
 \\ &  - \mathbb{E}_{o' \sim {\mathcal{P}}(\cdot|z_h, a_h)} [V_{h+1}^{\pi^*, \mathcal{M}, r}(z_{h+1}')] \bigr] - \mathcal{O}(H^2 \epsilon_1) - n \epsilon_{\mathrm{delay}},
\end{align}
where in the first inequality, we use the property of concentration of the bonus \cite[Lemma 16]{guo2023provably}, here $c$ is an absolute constant, and the lower bound of the bonus deviation follows from \eqref{eq:delay_bonus_bound}. The second inequality is by \eqref{eq:P_M_deviation}. We define
\begin{align}
g_h(z, a) & = \mathbb{E}_{o'_h \sim {\widehat{\mathcal{P}}}(\cdot|z, a)} [V_{h+1}^{\pi^*, \mathcal{M}, r}(c(z, a, o'_h))] 
\nonumber \\& - \mathbb{E}_{o'_h \sim {\mathcal{P}}(\cdot|z, a)} [V_{h+1}^{\pi^*, \mathcal{M}, r}(c(z, a, o'_h))].
\end{align}

With Lemma~\ref{lem:MLE}, for any $(z, a)$ we have
\begin{equation}
\mathbb{E}_{(z, a) \sim \rho_h} [g_h^2(z, a)] \le \zeta_k, \quad \mathbb{E}_{(z, a) \sim \beta_h} [g_h^2(z, a)] \le \zeta_k.
\end{equation}    

By Lemma~\ref{lem:L_step_learned_model}, we have
\begin{align} \label{eq:g_h_eror}
& \sum_{h=1}^{H}\mathbb{E}_{(z,a)\sim d^{\pi}_{\hat{\mathcal{P}},h}}
\!\left[g_h(z,a)\right] \nonumber
\\&\le \sum_{h=1}^{H}
\min\Bigl\{ 1, \mathbb{E}^{\hat{\mathcal{P}}}_{ z_{h-L-1},\,a_{h-L-1}\sim \pi}
\nonumber
\\& \bigl\|
\hat\phi^{\top}(z_{h-L-1},a_{h-L-1})
\bigr\|_{\rho_{h-L-1}^{-1},\,\hat\phi_{h-L-1}}
\Bigr\}
\nonumber\\& \cdot
\sqrt{ |A|^{L}k\zeta_k + 4\lambda_k d +4 k\epsilon_1 } + \mathcal{O}(H^2\epsilon_1)
\nonumber\\ &\le
\sum_{h=1}^{H}
\min\Bigl\{ 1, c\alpha_k
\mathbb{E}^{\hat{\mathcal M}}_{ z_{h-L-1},\,a_{h-L-1}\sim \pi}
\nonumber
\\& \bigl\| \hat\phi^{\top}(z_{h-L-1},a_{h-L-1})
\bigr\|_{\rho_{h-L-1}^{-1},\,\hat\phi_{h-L-1}}
\Bigr\}+ \mathcal{O}(H^2\epsilon_1),
\end{align}
where in the last step we use \eqref{eq:P_M_deviation} and the definition
\[\alpha_k=\sqrt{k|A|^{L}\zeta_k+4\lambda_kd+4 k\epsilon_1}/c. \]

For $h \le 0$, we have
\begin{equation} \label{eq:phi_1}
\|\hat \phi^{\top}(z_h,a_h) \|_{\rho_h^{-1},\hat\phi_h}
= \sqrt{\frac{1}{k+\lambda}}
< \frac{1}{\sqrt{k}},
\end{equation}
since $\phi(s,a)=e_1$ for $h \le 0$.

% [MODIFIED]: Added the delay penalty to the almost optimism bound
Combine \eqref{eq:m_error}, \eqref{eq:g_h_eror} and \eqref{eq:phi_1}, with probability $1-\delta$,
\begin{equation} \label{eq:V_M}
V^{\pi^\ast,\hat{\mathcal{M}},r+\tilde{b}_k}- V^{\pi^\ast,\mathcal{M},r}
\ge - \frac{c\alpha_k L}{\sqrt{k}} - \mathcal{O}(H^2\epsilon_1) - n \epsilon_{\mathrm{delay}}.
\end{equation}
    
Next, we start the analysis of the sample complexity as follows.
For a fixed $k$,
\begin{align}
&V^{\pi^\ast,\mathcal{M},r} - V^{\pi_k,\mathcal{M},r}
\nonumber
\\& \le V^{\pi^\ast,\mathcal{\hat{M}},r+\tilde{b}_{k}} - V^{\pi_k,\mathcal{M},r}
+ \frac{c\alpha_k L}{\sqrt{k}} + \mathcal{O}(H^2\epsilon_1) + n \epsilon_{\mathrm{delay}}
\nonumber \\
& \le V^{\pi_k,\hat{\mathcal M},r+\tilde{b}_{k}} - V^{\pi_k,\mathcal{M},r}
+ \frac{c\alpha_k L}{\sqrt{k}} + \mathcal{O}(H^2\epsilon_1) + n \epsilon_{\mathrm{delay}}
\nonumber \\&= \sum_{h=1}^{H}
\Biggl[ \mathbb E_{(z_h,a_h)\sim d_h^{\pi_k,\mathcal M}}
\!
\nonumber \\&
\left[ \tilde{b}_{h}(z_h,a_h) + \mathbb E_{o'_{h}\sim \hat {\mathcal M}(\cdot \mid z_h,a_h)} \![ V_{h+1}^{\pi_k,\hat{\mathcal M},r+\tilde{b}_{k}}(z'_{h+1}) ]
\right.
\nonumber \\&\left.
-
\mathbb E_{o'_{h}\sim {\mathcal M}(\cdot \mid z_h,a_h)}
[ V_{h+1}^{\pi_k,\hat{\mathcal M},r+\tilde{b}_{k}}(z'_{h+1}) ]
\right]
\Biggr]
\nonumber \\& + \frac{c\alpha_k L}{\sqrt{k}} + O(H^2\epsilon_1) + n \epsilon_{\mathrm{delay}},
\end{align}
where the first inequality comes from \eqref{eq:V_M}, the second inequality comes from the fact that with Algorithm~\ref{alg:ma-lsvi-jr},
\begin{equation}
\pi_k = \arg\max_{\pi} V^{\pi,\hat{\mathcal M},r+\tilde{b}_k},
\end{equation}
and the last equation comes from Lemma~\ref{lem:simulation_lemma}.

By \eqref{eq:P_M_deviation}, we further have
\begin{align}
&\sum_{h=1}^{H}
\Biggl[
\mathbb E_{(z_h,a_h)\sim d_h^{\pi_k,\mathcal M}}
\bigl[\tilde b_{h}(z_h,a_h) \nonumber
\\& + \mathbb E_{o'_{h}\sim \hat{\mathcal{M}}(\cdot \mid z_h,a_h)}
\![V_{h+1}^{\pi_k,  \hat{\mathcal{M}}, r+\tilde{b}_{k}}(z'_{h+1}) ]
\nonumber\\&
-\mathbb E_{o'_{h}\sim {\mathcal M}(\cdot \mid z_h,a_h)}
[ V_{h+1}^{\pi_k,\hat{\mathcal M},r+\tilde{b}_{k}}(z'_{h+1}) ]
 \bigr]
\Biggr]
+ \frac{c\alpha_k L}{\sqrt{k}} + n \epsilon_{\mathrm{delay}}
\nonumber \\& \le
\sum_{h=1}^{H}
\Biggl[
\mathbb E_{(z_h,a_h)\sim d_h^{\pi_k,\mathcal P}}
\bigl[\tilde b_{h}(z_h,a_h)\nonumber
\\& + \mathbb E_{o'_{h}\sim \hat {\mathcal{P}}(\cdot \mid z_h,a_h)}
[ V_{h+1}^{\pi_k,\hat{\mathcal M},r+\tilde{b}_{k}}(z'_{h+1}) ]
\nonumber \\& -\mathbb E_{o'_{h}\sim \mathcal P(\cdot \mid z_h,a_h)}
[ V_{h+1}^{\pi_k,\hat{\mathcal M},r+\tilde{b}_{k}}(z'_{h+1}) ] \bigr]
\Biggr] \nonumber \\
& + \frac{c\alpha_k L}{\sqrt{k}} + O(H^2\epsilon_1) + n \epsilon_{\mathrm{delay}}.
\end{align}

Denote
\begin{align}
&f_h(z_h,a_h)
\nonumber 
\\&= \frac{1}{2H+1}
\left(
\mathbb E_{o'_h\sim \hat {\mathcal{P}}(\cdot \mid z_h,a_h)}
\!\left[
V_{h+1}^{\pi_k,\hat{\mathcal M},r+\tilde{b}_k}(z'_{h+1})
\right] \nonumber
\right. \\& \left. -
\mathbb E_{o'_h\sim \mathcal{P}(\cdot \mid z_h,a_h)}
[ V_{h+1}^{\pi_k,\hat{\mathcal M},r + \tilde{b}_k}(z'_{h+1}) ]
\right).
\end{align}

Then we have
\begin{align} \label{eq:error_bounds}
% [MODIFIED]: Splitting \tilde{b} back into \widehat{b} and delay error
&V^{\pi^\ast,\mathcal M,r} - V^{\pi_k,\hat{\mathcal M}, r}
\nonumber \\& = \sum_{h=1}^{H}
\mathbb E_{(z_h,a_h)\sim d_h^{\pi_k,\mathcal P}}
\!\left[
\tilde b_h(z_h,a_h)
\right] \nonumber 
\\& + (2H+1)
\sum_{h=1}^{H}
\mathbb E_{(z_h,a_h)\sim d^{\pi_k,\mathcal P}}
\!\left[
f_h(z_h,a_h)
\right]
\nonumber \\& + \frac{c\alpha_k L}{\sqrt{k}}+ O(H^2\epsilon_1) + n \epsilon_{\mathrm{delay}}
\notag \\& \le
\sum_{h=1}^{H}
\mathbb E_{(z_h,a_h)\sim d^{\pi_k,\mathcal P}}
\!\left[ \widehat b_h(z_h,a_h) \right] + n \epsilon_{\mathrm{delay}} \nonumber
\\& + (2H+1)
\sum_{h=1}^{H}
\mathbb E_{(z_h,a_h)\sim d_h^{\pi_k,\mathcal P}}
\!\left[
f_h(z_h,a_h)
\right]
\notag\\
& +\frac{c\alpha_k L}{\sqrt{k}}+O(H^2\epsilon_1) + n \epsilon_{\mathrm{delay}}.
\end{align}

For the first term in \eqref{eq:error_bounds}, since it is now the ideal bonus $\widehat{b}_h$, we have
\begin{align}
&\sum_{h=1}^{H}\mathbb{E}_{(z_h,a_h)\sim d_h^{\pi_k,\mathcal P}}
\!\left[\widehat b_h(z_h,a_h)\right] \nonumber
\\&\le
\sum_{h=0}^{H}
\mathbb{E}_{(\tilde z,\tilde a)\sim d_{h-L}^{\pi_k,\mathcal P}}
\Bigl[
\|\phi^\ast_{h-L}(z,\tilde a)\|_{\Sigma^{-1}_{\gamma_{h-L}},\,\hat\phi^\ast_{h-L}} \nonumber 
\\& \sqrt{
k|A|^{L}\,
\mathbb E_{(z,a)\sim \rho_h}\!\left[(\widehat b_h(z,a))^2\right]
+ 4\lambda_k d + 4 k\epsilon_1
}
\Bigr] + 2H \epsilon_1,
\end{align}
where the inequality follows from Lemma~\ref{lem:L_step_true_model} associated with $\|\widehat b_h\|_\infty \le 1$.
% Note that we use the fact that $B=2$ when applying Lemma~\ref{lem:L_step_true_model}. 
In addition, we have that for any $h\in[H]$,
\begin{align}
&k\mathbb E_{(z,a)\sim \rho_h}
\left[
\|\hat\phi_h(z,a)\|_{\Sigma^{-1}_{\rho_h},\,\hat\phi_h}^{2}
\right] \nonumber
\\& = k\,\mathrm{Tr}\left(
\mathbb E_{\rho_h}[\hat\phi_h \hat\phi_h^\top]
\left[ k\mathbb E_{\rho_h}[\hat\phi_h \hat \phi^\top_h]+\lambda_k I
\right]^{-1}
\right)
\le d.
\end{align}

Then we have
\begin{align}
&\sum_{h=1}^{H}\mathbb{E}_{(z,a)\sim d_h^{\pi_k,\mathcal P}}
\!\left[\widehat b(z,a)\right] \nonumber
\\&\le
\sum_{h=1}^{H}
\mathbb{E}_{(\tilde z,\tilde a)\sim d_{h-L}^{\pi_k,\mathcal P}}
\!\left[
\|\phi_{h-L}^\ast(z,\tilde a)\|_{\Sigma^{-1}_{\rho_{h-L}},\,\hat\phi^\ast_{h-L}}
\right] \nonumber
\\&
\sqrt{|A|^{L}\alpha_k^{2}d + 4\lambda_k d + 4 k\epsilon_1} + 2H\epsilon_1.
\end{align}

Next, we bound the second term in~\eqref{eq:error_bounds}, with Lemma~\ref{lem:L_step_true_model} and $\|f_h(z,a)\|_\infty \le 1$, we have
\begin{align}
&\sum_{h=1}^{H}\mathbb{E}_{(z_h,a_h)\sim d_h^{\pi_k,\mathcal P}} \left[f_h(z_h,a_h)\right] \nonumber
\\& \le \sum_{h=1}^{H} \mathbb{E}_{(\tilde z,\tilde a)\sim d_{h-L}^{\pi_k,\mathcal P}}
\Bigl[
\|\phi^\ast_{h-L}(z,\tilde a)\|_{\Sigma^{-1}_{\gamma_{h-L}},\,\hat\phi^\ast_{h-L}}
\nonumber \\&
\sqrt{ k|A|^{L}
\mathbb E_{(z,a)\sim \rho_h}\!\left[f_h^2(z,a)\right] + 4\lambda_k d
+ 4k\epsilon_1
} \Bigr]
\nonumber \\& \le
\sum_{h=1}^{H}
\mathbb{E}_{(\tilde z,\tilde a)\sim d_{h-L}^{\pi_k,\mathcal P}}
\left[
\|\phi^\ast_{h-L}(z,\tilde a)\|_{\Sigma^{-1}_{\gamma_{h-L}},\,\hat\phi^\ast_{h-L}}
\right]
\nonumber \\& \sqrt{k|A|^{L}\zeta_k + 4\lambda_k d +4 k\epsilon_1},
\end{align}
where in the second inequality, we use 
\[
\mathbb E_{(z,a)\sim \rho_h}[f_h^2(z,a)] \le \zeta_k.
\]

Then we have
\begin{align}
&V^{\pi^\ast,\mathcal M,r} - V^{\pi_k, {\mathcal M},r}
\nonumber \\&\le
\sum_{h=1}^{H}
\mathbb E_{(z_h,a_h)\sim d_h^{\pi_k,\mathcal M}}
\!\left[\widehat b(z_h,a_h)\right]
\nonumber
\\&+
(2H+1)\sum_{h=1}^{H}
\mathbb E_{(z_h,a_h)\sim d_h^{\pi_k,\mathcal M}}
\!\left[f_h(z_h,a_h)\right]
\notag \\& +
\frac{c\alpha_k L}{\sqrt{k}}+
O(H^2\epsilon_1) + 2n \epsilon_{\mathrm{delay}}
\notag\\&\le
\sum_{h=1}^{H}
\mathbb E_{(\tilde z,\tilde a)\sim d_{h-L}^{\pi_k,\mathcal P}}
\!\left[\|\phi_{h-L}^\ast(\tilde z,\tilde a)\|_{\Sigma^{-1}_{\gamma_{h-L}},\hat\phi_{h-L}^\ast}
\right] \nonumber
\\&
\sqrt{|A|^{L}\alpha_k^{2}d + 4\lambda_k d + 4 k\epsilon_1}
\nonumber \\&+(2H+1) \sum_{h=1}^{H}
\mathbb E_{(\tilde z,\tilde a)\sim d_{h-L}^{\pi_k,\mathcal P}}
\!\left[\|\phi_{h-L}^\ast(\tilde z,\tilde a)\|_{\Sigma^{-1}_{\gamma_{h-L}},\hat\phi_{h-L}^\ast}
\right] \nonumber\\&
\sqrt{k|A|^{L}\zeta_k + 4\lambda_k d +4 k\epsilon_1}
+\frac{c\alpha_k L}{\sqrt{k}}+
O(H^2\epsilon_1) + 2n \epsilon_{\mathrm{delay}}.
\end{align}

Next, we bound the first two terms.
Note that $\gamma_h^k(z,a)=\frac{1}{k}\sum_{j=0}^{k-1} d_h^{\pi_j}(z,a),$ then
\begin{align}
&    \sum_{k=1}^{K}
\mathbb E_{(\tilde z,\tilde a)\sim d_h^{\pi_k,\mathcal P}}
\!\left[ \phi^\ast_h(\tilde z,\tilde a)^{\top}
\Sigma^{-1}_{\gamma_h^k} \phi^\ast_h(\tilde z,\tilde a)
\right] 
\nonumber
\\& 
\le   (
\log\det (
\sum_{k=1}^{K}
\mathbb E_{(\tilde z,\tilde a)\sim d_h^{\pi_k,\mathcal P}}
\!\left[
\phi_h^\ast(\tilde z,\tilde a)\phi_h^\ast(\tilde z,\tilde a)^{\top}
\right] ) \nonumber
\\&-\log\det(\lambda I)
) 
 \notag
\\& \le
d\log\left(1+\frac{K}{d\lambda_1}\right),
\end{align}
% \begin{align}
% &\sum_{k=1}^{K}
% \mathbb E_{(\tilde z,\tilde a)\sim d_h^{\pi_k,\mathcal P}}
% \!\left[
% \|\phi^\ast_h(\tilde z,\tilde a)\|_{\Sigma^{-1}_{\gamma_h^k},\,\phi^\ast_h}
% \right] \nonumber
% \\& \le
% \sqrt{  (  K \sum_{k=1}^{K}
% \mathbb E_{(\tilde z,\tilde a)\sim d_h^{\pi_k,\mathcal P}}
% \!\left[ \phi^\ast_h(\tilde z,\tilde a)^{\top}
% \Sigma^{-1}_{\gamma_h^k} \phi^\ast_h(\tilde z,\tilde a)
% \right] ) }
% \nonumber
% \\& 
% \le  \Bigl( K (
% \log\det (
% \sum_{k=1}^{K}
% \mathbb E_{(\tilde z,\tilde a)\sim d_h^{\pi_k,\mathcal P}}
% \!\left[
% \phi_h^\ast(\tilde z,\tilde a)\phi_h^\ast(\tilde z,\tilde a)^{\top}
% \right] ) \nonumber
% \\&-\log\det(\lambda I)
% ) \Bigr)^{1/2}
%  \notag
% \\& \le\sqrt{
% dK\log\!\left(1+\frac{K}{d\lambda_1}\right)
% },
% \end{align}
where the first inequality is by Lemma~\ref{lem:semidefinite_matrix}, and the second inequality is by Lemma~\ref{lem:matrix_eigenvalue}.
 Then using  the Cauchy-Schwarz inequality,
\begin{align}
&\sum_{k=1}^K\sum_{h=1}^{H}
\mathbb E_{(\tilde z,\tilde a)\sim d_{h-L}^{\pi_k,\mathcal P}}
\!\left[\|\phi_{h-L}^\ast(\tilde z,\tilde a)\|_{\Sigma^{-1}_{\gamma_{h-L}},\hat\phi_{h-L}^\ast}
\right] \nonumber
\\& 
\sqrt{ |A|^{L}\alpha_k^{2}d + 4\lambda_k d + 4 k\epsilon_1} \nonumber
\\& \leq H \sqrt{ \bigl(\sum_{k=1}^K (|A|^{L}\alpha_k^{2}d + 4\lambda_k d + 4 k\epsilon_1) \bigr) d\log\left(1+\frac{K}{d\lambda_1}\right) }.
\end{align}
Given the parameters in~\eqref{eq:parameters}, we obtain that
\begin{align}
&\bigl(\sum_{k=1}^K \sum_{h=1}^{H}(|A|^{L}\alpha_k^{2}d + 4\lambda_k d + 4 k\epsilon_1) \bigr)d\log\left(1+\frac{K}{d\lambda_1}\right)
\nonumber
\\& \le
O\bigl(H (( \frac{|A|^{2L}d^2}{H-n} + |A|^L d^4 + d^3 )
\nonumber\\&
K\log(dK(H-n)|\mathcal F|/\delta)
\log (1+\frac{K}{d\lambda_1} )
)^{\frac{1}{2}} \bigr).
\end{align}
 
Combining all of the above relations and using~\eqref{eq:bouns_error}, let $\Lambda_K=\log\left(\frac{dK(H-n)|\mathcal F|}{\delta}\right)$, we have 
\begin{align}
&\sum_{k=1}^{K}
( V^{\pi^\ast,\mathcal M,r}-V^{\pi_k,\mathcal M,r} )
% \nonumber \\&\le
% O\bigl(H^2 (( \frac{|A|^{2L}d^2}{H-n} + |A|^L d^4 + d^3 )
% \nonumber\\&
% K\log(dK(H-n)|\mathcal F|/\delta)
% \log (1+\frac{K}{d\lambda_1} )
% )^{\frac{1}{2}} \bigr)
%  + O(H^2 K\epsilon_1) \nonumber
% \\&+ \sum_{k=1}^K 2n\alpha_k \sqrt{ \frac{2L_\phi L_\psi C_b}{\lambda_k}
% + \frac{2k L_\phi^3L_\psi C_b}{\lambda_k^2}
% + \frac{L_\phi^2 C_D}{\lambda_k^2}
% }  +\frac{cK \alpha_k L}{\sqrt{k}}
% \nonumber
% \\& + \sum_{k=1}^K\frac{c \alpha_k L}{\sqrt{k}}
\nonumber  \\&
\le O\bigl(H^2 (( \frac{|A|^{2L}d^2}{H-n} + |A|^L d^4 )
\nonumber\\&
K\log(dK(H-n)|\mathcal F|/\delta)
\log (1+\frac{K}{d\lambda_1} )
)^{\frac{1}{2}} \bigr)
 + O(H^2 K\epsilon_1) \nonumber \\&
 + \sum_{k=1}^K 2n\alpha_k
\sqrt{\frac{2L_\phi L_\psi C_b+2L_\phi^2}{\lambda_k}
} + \sum_{k=1}^K\frac{c \alpha_k L}{\sqrt{k}} \nonumber
\\& \leq O\Bigg(
H^2\sqrt{ \left( \frac{|A|^{2L}d^2}{H-n} + |A|^L d^4 \right)
K\Lambda_K \log K
}\Bigg)
\nonumber\\
&\quad+ O\Bigg( nK \sqrt{ \frac{|A|^L}{d(H-n)}+d }\Bigg)
\nonumber\\
&\quad+ O\!\left(
L \sqrt{ (\frac{|A|^L}{H-n}+d^2 )K\Lambda_K}\right)+ O(H^2 K\epsilon_1).
\end{align}

%%%%%%%%%%%%%%%%%%%%%%%%%%%%%%%%%%%%%%%%%%%%%%%%%%%%%%%%%%%%%%%%%
Then, based on \eqref{eq:V_P_M_deviation}, it can be concluded that with probability $1-\delta$,
\begin{align} \label{eq:V_error}
&\sum_{k=1}^K ( V^{\pi^*, {P}, r} - V^{\pi^k, {P}, r}  )
\nonumber 
\\& \leq O\Big( H^2\sqrt{( \frac{|A|^{2L}d^2}{H-n} + |A|^L d^4) K\log(\frac{dK(H-n)|\mathcal F|}{\delta}) \log K }\Big)
\nonumber\\
&
+ O\Big( L \sqrt{ (\frac{|A|^L}{H-n}+d^2 )K\log(\frac{dK(H-n)|\mathcal F|}{\delta})}\Big)
\nonumber\\
&+ O\Big( nK \sqrt{ \frac{|A|^L}{d(H-n)}+d }\Big)+ O(H^2 K\epsilon_1),
\end{align}
substituting the definition of $\epsilon_1$ and dividing \eqref{eq:V_error} by $K$, the conclusion follows.

\bibliographystyle{IEEEtran}
\bibliography{reference,IDS}

\end{document}